\documentclass[11pt]{article}
\usepackage{amsmath}
\usepackage{amsfonts}
\usepackage{amssymb}
\usepackage{mathrsfs}
\usepackage{amsthm}
\usepackage{palatino}
\usepackage[margin=2.5cm, vmargin={1.5cm}]{geometry}

\usepackage{dcolumn}

\usepackage{times}
\usepackage{graphicx}
\usepackage{xcolor}

\usepackage{pstricks}
\usepackage{pst-plot}

\renewcommand\footnotemark{}

\begin{document}

\title{Symmetric functions and a natural framework for combinatorial and number theoretic sequences}

\author{
  Cormac ~O'Sullivan\footnote{{\it Date:} Mar 6, 2022.
\newline \indent \ \ \
  \newline \indent \ \ \
Support for this project was provided by a PSC-CUNY Award, jointly funded by The Professional Staff Congress and The City
\newline \indent \ \ \
University of New York.}
  }

\date{}

\maketitle

%modular symbol
\def\s#1#2{\langle \,#1 , #2 \,\rangle}

%domains
\def\H{{\mathbf{H}}}
\def\F{{\frak F}}
\def\C{{\mathbb C}}
\def\R{{\mathbb R}}
\def\Z{{\mathbb Z}}
\def\Q{{\mathbb Q}}
\def\N{{\mathbb N}}
%symbols
\def\G{{\Gamma}}
\def\GH{{\G \backslash \H}}
\def\g{{\gamma}}
\def\L{{\Lambda}}
\def\ee{{\varepsilon}}
\def\K{{\mathcal K}}
\def\Re{\mathrm{Re}}
\def\Im{\mathrm{Im}}
\def\PSL{\mathrm{PSL}}
\def\SL{\mathrm{SL}}
\def\Vol{\operatorname{Vol}}
\def\lqs{\leqslant}
\def\gqs{\geqslant}
\def\sgn{\operatorname{sgn}}
\def\res{\operatornamewithlimits{Res}}
\def\li{\operatorname{Li_2}}
\def\lip{\operatorname{Li}'_2}
\def\pl{\operatorname{Li}}

\def\clp{\operatorname{Cl}'_2}
\def\clpp{\operatorname{Cl}''_2}
\def\farey{\mathscr F}

\def\dm{{\mathcal A}}
\def\ov{{\overline{p}}}
\def\ja{{K}}

% this fixes spacing around \left, \right
\let\originalleft\left
\let\originalright\right
\renewcommand{\left}{\mathopen{}\mathclose\bgroup\originalleft}
\renewcommand{\right}{\aftergroup\egroup\originalright}

\newcommand{\stira}[2]{{\genfrac{[}{]}{0pt}{}{#1}{#2}}}
\newcommand{\stirb}[2]{{\genfrac{\{}{\}}{0pt}{}{#1}{#2}}}
\newcommand{\stirc}[2]{{\genfrac{\|}{\|}{0pt}{}{#1}{#2}}}

\newcommand{\norm}[1]{\left\lVert #1 \right\rVert}

\newcommand{\e}{\eqref}

%theorems

\newtheorem{theorem}{Theorem}[section]
\newtheorem{lemma}[theorem]{Lemma}
\newtheorem{prop}[theorem]{Proposition}
\newtheorem{conj}[theorem]{Conjecture}
\newtheorem{cor}[theorem]{Corollary}
\newtheorem{assume}[theorem]{Assumptions}
\newtheorem{adef}[theorem]{Definition}
\newtheorem{eg}[theorem]{Example}

%\newcounter{coundef}
%\newtheorem{adef}[coundef]{Definition}

\newcounter{counrem}
\newtheorem{remark}[counrem]{Remark}

\renewcommand{\labelenumi}{(\roman{enumi})}
\newcommand{\spr}[2]{\sideset{}{_{#2}^{-1}}{\textstyle \prod}({#1})}
\newcommand{\spn}[2]{\sideset{}{_{#2}}{\textstyle \prod}({#1})}

\numberwithin{equation}{section}

\bibliographystyle{alpha}

\begin{abstract}
Certain triples of power series, considered by I.\ Macdonald, give a natural framework for many combinatorial and number theoretic sequences, such as the Stirling, Bernoulli and harmonic numbers and partitions of different kinds. The power series in such a triple are closely linked by identities coming from the theory of symmetric functions. We extend the work of  Z-H.\ Sun, who developed similar ideas, and Macdonald, revealing more of the structure of these triples. De Moivre polynomials  play a key role in this study.% and their properties are developed further.
\end{abstract}

%\tableofcontents

\section{Introduction}

We begin with an example due to E.T. Bell from \cite{Bell27} that will help introduce the main ideas. It is based on
Euler's pentagonal number theorem, which states that
\begin{equation} \label{el}
  \prod_{j=1}^\infty (1-q^j) =  \sum_{m=0}^\infty c(m) q^m \qquad \text{for} \qquad
  c(m) = \begin{cases}
  (-1)^r & \text{ if $m=\frac{r(3r-1)}2$ for  $r\in \Z$},\\
  0 & \text{ otherwise. }
  \end{cases}
\end{equation}
Expanding the logarithm of the product shows
\begin{equation} \label{bex}
  \log \prod_{j=1}^\infty (1-q^j) = \sum_{j=1}^\infty \log (1-q^j)
  = -\sum_{j=1}^\infty \sum_{k=1}^\infty \frac{q^{kj}}{k} = -\sum_{m=1}^\infty  \frac{\sigma(m)}{m} q^{m},
\end{equation}
where $\sigma(m)$ indicates the sum of the divisors of $m$. Then exponentiating \e{bex} and comparing coefficients of powers of $q$ gives the equalities
\begin{equation} \label{bex2}
  c(1)= -\sigma(1), \qquad c(2)= \frac {\sigma(1)^2}2  - \frac {\sigma(2)}2, \qquad
  c(3)= -\frac {\sigma(1)^3}6  + \frac {\sigma(1)\sigma(2)}2 -\frac {\sigma(3)}3,
\end{equation}
and, in general, $c(n)$ is a degree $n$ polynomial in $\sigma(1), \dots, \sigma(n)$. Bell called these  {\em partition polynomials}, presumably since they have one term for each of the $p(n)$ partitions of the integer $n$. See \e{bepr2} for his general definition.

One aspect of this is quite ancient and
 arises whenever  polynomials or power series are raised to powers.

\begin{adef} \label{dbf}
{\rm For integers $n$, $k$ with $k\gqs 0$, the {\em De Moivre polynomial} $\dm_{n,k}(a_1, a_2, \dots)$ is defined by
\begin{equation} \label{bell}
    \left( a_1 x +a_2 x^2+ a_3 x^3+ \cdots \right)^k = \sum_{n\in \Z} \dm_{n,k}(a_1, a_2, a_3, \dots) x^n \qquad \quad (k \in \Z_{\gqs 0}).
\end{equation}}
\end{adef}

These polynomial coefficients $\dm_{n,k}$  may be
  traced back  to
  De Moivre's paper \cite{dem} from 1697, and their properties are gathered in \cite{odm}.
If $n<k$ then clearly $\dm_{n,k}(a_1, a_2, a_3, \dots)=0$. If $n\gqs k$ then
 \begin{equation} \label{bell2}
  \dm_{n,k}(a_1, a_2, a_3, \dots) = \sum_{\substack{1j_1+2 j_2+ \dots +mj_m= n \\ j_1+ j_2+ \dots +j_m= k}}
 \binom{k}{j_1 , j_2 ,  \dots , j_m} a_1^{j_1} a_2^{j_2}  \cdots a_m^{j_m},
\end{equation}
where $m=n-k+1$, the sum   is over all possible $j_1$, $j_2$,  \dots , $j_m \in \Z_{\gqs 0}$ and $0^0$ always means $1$. It is a polynomial in $a_1, a_2, \dots, a_{m}$ of homogeneous degree $k$ with positive integer coefficients.

%(Note that $m$ above is as large as possible since $1j_2+2j_3+ \cdots +(m-1)j_m =n-k$.)
%(The convention $0^0=1$ is used throughout.)
In this notation the general form of \e{bex2} is
\begin{equation}\label{bex3}
  c(n)  = \sum_{k=0}^{n}  \frac{(-1)^k}{k!} \dm_{n,k}\left(\frac{\sigma(1)}{1}, \frac{\sigma(2)}{2}, \frac{\sigma(3)}{3}, \cdots \right).
\end{equation}
This is equivalent to Bell's formulation (see the first example in \cite[p. 44]{Bell27}) and we will see that it follows from \e{cox}.
Besides $c(n)$ and $\sigma(n)$, there is an obvious third sequence to consider here since
\begin{equation}\label{bpart}
  F(q):= \prod_{j=1}^\infty \frac 1{1-q^j} = \sum_{n=0}^\infty p(n) q^n
\end{equation}
is the generating function for the partitions. A similar formula to \e{bex3} with $p(n)$ on the left can be found, as well as the well-known  recurrences,
\begin{align}\label{prec}
  p(n) & = -\sum_{k=0}^{n-1} p(k) c(n-k), \\
  p(n) & = \frac 1n \sum_{k=0}^{n-1} p(k) \sigma(n-k), \label{prec2}
\end{align}
for $n\gqs 1$. %, linking $p(n)$ to $c(n)$ and $\sigma(n)$.
Here \e{prec} is due to Euler and comes from multiplying \e{el} by \e{bpart} and equating coefficients, while \e{prec2} comes from equating $-F'(q)/F(q)$ and the derivative of \e{bex}.

I.\ Macdonald related the sequences  $c(n)$, $p(n)$ and $\sigma(n)$   to symmetric functions in \cite[Ex. 6, p. 27]{Macd} and the identities we have seen follow easily from the known relations among these functions. He gave many other examples   in \cite[pp. 26 -- 36]{Macd} and they lead naturally   to the following simple abstract definition.

\begin{adef} \label{mcd}
{\rm Let $E(t)$, $H(t)$ and $P(t)$ be three formal power series with $E(t)$ and $H(t)$  having constant term $1$. We may call them a {\em symmetric triple} if they satisfy
\begin{equation} \label{pp}
  P(-t)= \frac d{dt}\log E(t), \qquad P(t)= \frac d{dt}\log H(t).
\end{equation}}
\end{adef}

 Independently, Z-H.\ Sun  studied  {\em Newton-Euler pairs} in \cite{Sun05} -- these are the  $(H(t),P(t))$ part. We will look at almost all of the symmetric triples contained in Macdonald and Sun's  work, find new cases, and reveal more of their overall structure; these triples form a group for instance.

Following Macdonald, it is convenient to write the coefficients of the three series as
\begin{equation}\label{ehpser}
E(t)=\sum_{n=0}^\infty e_{n} t^n,  \qquad H(t)=\sum_{n=0}^\infty h_{n} t^n,   \qquad P(t)=\sum_{n=0}^\infty p_{n+1} t^n,
\end{equation}
noting that the constant term of $P(t)$ is labelled $p_1$, and  $p_0$  is left undefined. In this language our Bell/Euler example is the following symmetric triple. (All Examples \ref{bell-eu} to \ref{logxs} are symmetric triples so we will not always say it explicitly. Also it is understood that $e_n$ and $h_n$ are given for $n\gqs 0$ while $p_n$ is given for $n\gqs 1$.)
\begin{eg}[Partitions, divisor sums] \label{bell-eu}
\begin{alignat*}{3}%\label{xexp}
  E(t) & = \prod_{j=1}^\infty (1-(-t)^j), & \qquad H(t) & =\prod_{j=1}^\infty \frac 1{1-t^j}, & \qquad P(t) & =\sum_{j=1}^\infty \frac {j \cdot t^{j-1}}{1-t^j}, \\
  \text{with} \qquad
  e_n & = (-1)^n c(n), & \qquad  h_n & = p(n), & \qquad  p_n & =  \sigma(n).
\end{alignat*}

\end{eg}

As we will see, the coefficients $e_n$, $h_n$ and $p_n$ of a symmetric triple are tightly linked, with five convolution identities connecting them, similar to \e{prec}, \e{prec2}, and six transition formulas, similar to \e{bex3}, which correspond to changing the symmetric function basis.

To give a second example based on \cite{Bell27}, define $N_r(n)$ to be the number of ways to express $n$ as a sum of squares of $r$ integers (including the order of the integers and their signs), and let $\omega_2(n)$  be the sum of the odd divisors of $n$. The details of the next symmetric triple, including the $\overline{p}_r(n)$ term, are explained in section \ref{sq-tri}.

\begin{eg}[Representations of sums of squares] \label{bell-eu2} For a positive integer $r$,
\begin{gather*}%\label{xexp}
  e_n = N_r(n),
  \qquad h_n = \overline{p}_r(n), \qquad p_n =  r (\sigma(n)+\omega_2(n)).
\end{gather*}
\end{eg}
\begin{proof}[Discussion]
 Bell's formula for the number of representations of $n$ as a sum of $r$ squares is equivalent to
\begin{equation}\label{bex4}
  N_r(n)  = \sum_{k=0}^{n}  (-1)^{n-k} \frac{r^{k}}{k!} \dm_{n,k}\left(\frac{\sigma(1)+\omega_2(1)}{1}, \frac{\sigma(2)+\omega_2(2)}{2},  \cdots \right). %\frac{\sigma(3)+\omega_2(3)}{3},
\end{equation}
See \cite[p. 46]{Bell27} for this and generalizations. We used the following simplification of \cite[Eq. (13)]{Bell27}:
$$(2(-1)^n -1)\omega_2(n)+\sigma(n) = (-1)^n(\sigma(n)+\omega_2(n)).$$
The identity \e{bex4} may be compared with Jacobi's  evaluations of $N_2(n)$ and $N_4(n)$ in terms of divisor sums, contained in \cite[Thm. 10.6.1]{AAR}.
The case $r=3$ of \e{bex4} gives class numbers of imaginary quadratic fields; see   \cite{Mort}, for instance, for the connection between $N_3(n)$ and class numbers.
\end{proof}

The plan of this paper is as follows.
Section \ref{dmo}  reviews material we will need on De Moivre polynomials. Sections \ref{sympo} and \ref{trpss} then  describe the relevant aspects of symmetric polynomials and functions, and explain the properties of symmetric triples, including how to multiply, divide and take arbitrary powers of them. The remaining sections look at a wide variety of examples, including those relating to  Bernoulli, Catalan, Cauchy and Stirling numbers, $q$-series, partitions and various polynomial families. We find this setup gives a natural organizing structure that allows these sequences to be studied systematically. New triples also often suggest themselves.

Recall that the Stirling subset numbers $\stirb{n}{k}$  count the number of ways to partition  $n$ elements into $k$ nonempty subsets. The Stirling cycle numbers $\stira{n}{k}$ count the number of ways to arrange $n$ elements into $k$ cycles.
Of particular interest is a similar array of numbers that naturally appears when $p_k$ in \e{ehpser} is set to equal
\begin{equation}\label{hnk}
  H_n^{(k)} :=\frac 1{1^k}+\frac 1{2^k}+\cdots +\frac 1{n^k},
\end{equation}
the harmonic number of order $k$ with $n$ summands, (with $H_n$ also meaning $H_n^{(1)}$). Then $h_k$ can be thought of as a type of harmonic multiset coefficient and we introduce the notation $\stirc{n}{k}$ for it. Unexpectedly, it includes both kinds of Stirling numbers as special cases; see Example \ref{onic}.

In section 8 we show that repeatedly taking powers and compositional inverses of a power series results in a two-parameter family. This allows us to generalize some of the earlier symmetric triple examples.

\section{De Moivre polynomials} \label{dmo}

\subsection{Basic properties}

  As in \cite[Sect. 2]{odm}, for $n,k\in \Z$ with $k\gqs 0$ we have the  relations
\begin{align}\label{esb}
  \dm_{n,0}(a_1, a_2, a_3,   \dots) & = \delta_{n,0}, \\
  \dm_{n,1}(a_1, a_2, a_3,  \dots) & = a_n \qquad \qquad (n\gqs 1), \label{esb2} \\
  \dm_{n,k}(c a_1, c a_2, c a_3, \dots) & = c^k \dm_{n,k}(a_1, a_2, a_3, \dots), \label{mulk}\\
  \dm_{n,k}(c a_1, c^2 a_2, c^3 a_3, \dots) & = c^n \dm_{n,k}(a_1, a_2, a_3, \dots). \label{muln}
\end{align}
By a power series, in this paper, we mean a formal power series $a_0+a_1 x+ a_2 x^2+\dots$ with coefficients $a_j$ in an integral domain $R$ that contains $\Q$. These series form another integral domain denoted $R[[x]]$.  The next basic result is \cite[Prop. 3.1]{odm}.

\begin{prop} \label{comp}
Suppose that $f(x)=a_1 x+a_2 x^2+ \cdots$ and $g(x)=b_0+b_1 x+b_2 x^2+ \cdots$ are two power series in $R[[x]]$. Then
$
  g(f(x))=c_0+c_1 x+c_2 x^2+ \cdots $ is in $R[[x]]$
with
\begin{equation}\label{esto}
 \quad c_n = \sum_{k=0}^n b_k  \cdot \dm_{n,k}(a_1,a_2,\dots).
\end{equation}
\end{prop}

%From now on we assume that the integral domain $R$ contains $\Q$.
Applying Proposition \ref{comp} with $g(x)=(1+x)^\alpha$, $e^{\alpha x}$ and $\log(1+\alpha x)$ gives
\begin{align}\label{pot}
  [x^n] (1+f(x))^\alpha & = \sum_{k=0}^n \binom{\alpha}{k}  \dm_{n,k}(a_1,a_2,\dots),\\
  [x^n] e^{\alpha f(x)} & = \sum_{k=0}^n \frac {\alpha^k}{k!} \dm_{n,k}(a_1,a_2,\dots),\label{cox}\\
  [x^n] \log(1+ \alpha f(x)) & = \sum_{k=1}^n (-1)^{k-1}\frac{\alpha^k}{k} \dm_{n,k}(a_1,a_2,\dots), \label{log}
\end{align}
where $[x^n]$ extracts  the  coefficient of $x^n$ in each series.
%The polynomials in $a_1, a_2, \dots$ on the right sides of \e{pot}, \e{cox} and \e{log} are essentially the potential, complete exponential and logarithmic polynomials, respectively, of \cite[Sections 3.3, 3.5]{Comtet}.
Recall that the general binomial coefficients satisfy $\binom{\alpha}{0}:= 1$  and
\begin{equation}\label{binx}
  \binom{\alpha}{k}:= \frac{\alpha(\alpha-1) \cdots (\alpha-k+1)}{k!}, \qquad  \binom{-\alpha}{k}=(-1)^k\binom{\alpha+k-1}{k}
\end{equation}
for  positive integers $k$ and variable $\alpha$. So the right sides of \e{pot}, \e{cox} and \e{log} are also degree $n$ polynomials in $\alpha$.
Therefore the series $(1+f(x))^\alpha$, $e^{\alpha f(x)}$ and $\log(1+ \alpha f(x))$ make sense for arbitrary $\alpha$,  giving elements of $R[\alpha][[x]]$.

De Moivre polynomials were introduced in \cite{dem}. They were rediscovered and appear explicitly in Arbogast's work from 1800 on higher derivatives of compositions, as discussed in \cite[Sect. 3]{odm}, and in several of Stern's  papers including \cite{Ste} from 1843; see \e{ste}, \e{ste2}. De Moivre polynomials also appear implicitly in formulas of Waring \e{war}, Kramp \e{kr}, \e{kr2} and Bell \e{bepr2}. %The general idea has probably been rediscovered many times.
In section 3.3 of Comtet's influential book \cite{Comtet} they are briefly described as `partial ordinary Bell polynomials'. The focus there is on the related `partial Bell polynomials' which include extra factorial terms that are usually not needed.

%\begin{prop} \label{cinvx}
%Suppose  $f(x)=a_1 x+a_2 x^2+ \cdots$ in $R[[x]]$ has  compositional inverse $b_1 x+b_2 x^2+ \cdots$. Then for positive integers $m$, $r$,
%\begin{align}\label{lag}
%  b_m = \frac 1{m} \sum_{k=0}^{m-1} \binom{-m}{k} a_1^{-m-k} \dm_{m-1,k}(a_2,a_3,\dots),\\
%  \dm_{m,r}(b_1,b_2, \dots) = \frac r{m} \sum_{k=0}^{m-r} \binom{-m}{k} a_1^{-m-k} \dm_{m-r,k}(a_2,a_3,\dots).
%  \label{lagxx}
%\end{align}
%\end{prop}

\subsection{Determinant formulas}

Define the $n \times n$ matrix
\begin{equation*}
  \mathcal{M}_n(t) := \left(
  \begin{matrix}
  a_1 t & 1 & &  \\
  a_2 t & a_1 t & 1 &  \\
  \vdots & & \ddots &  \\
  a_n t & a_{n-1} t & \dots &  a_1 t
  \end{matrix}
  \right),
\end{equation*}
with ones above the main diagonal and zeros above those. 
Also set
\begin{equation*}
  \mathcal{N}_n(t) := \left(
  \begin{matrix}
  a_1 t & 1 & & & \\
  a_2 t & a_1 t & 2 & & \\
  a_3 t & a_2 t & a_1 t & 3 & \\
  \vdots & & & \ddots &  \\
  a_n t & a_{n-1} t & \dots & a_2 t & a_1 t
  \end{matrix}
  \right),
  \qquad
  \mathcal{O}_n(t) := \left(
  \begin{matrix}
   a_1 t & 1 & & & \\
  2a_2 t & a_1 t & 1 & & \\
  3a_3 t & a_2 t & a_1 t & 1 & \\
  \vdots & & & \ddots &  \\
  n a_n t & a_{n-1} t & \dots & a_2 t & a_1 t
  \end{matrix}
  \right),
\end{equation*}
which are the same as $\mathcal{M}_n(t)$ except for multiplying by a factor $j$ on row $j$, above the main diagonal in $\mathcal{N}_n(t)$ and in the first column for $\mathcal{O}_n(t)$.

We will need the next identities in section \ref{sympo}. They are  inspired by \cite[Ex. 8]{Macd} with
\e{far} also giving a version of \cite[Thm. 3.1]{EJ}. See \cite[Sect. 5]{odm} for their proofs.

\begin{prop} \label{detp}
We have
\begin{align}
\label{far}
  \sum_{k=0}^n (-1)^{n+k} t^k \dm_{n,k}(a_1,a_2, \dots) & = \det \mathcal{M}_n(t)\\
  \sum_{k=0}^n (-1)^{n+k} \frac{t^k}{k!} \dm_{n,k}\left(\frac{a_1}1,\frac{a_2}2, \dots\right) & = \frac{1}{n!}\det \mathcal{N}_n(t),\label{far2}\\
  \sum_{k=0}^n (-1)^{n+k} \frac{t^k}{k} \dm_{n,k}\left(a_1,a_2, \dots\right) & = \frac{1}{n}\det \mathcal{O}_n(t).
  \label{far3}
\end{align}
\end{prop}

\subsection{Stirling numbers and further results}

%Some identities we will need are developed next.
As in \cite[Sect. 2]{odm},  generating function arguments show,   for $n\gqs k\gqs 0$,
\begin{equation}\label{bino}
  \dm_{n,k}\left( \binom{\alpha}{0}, \binom{\alpha}{1},\binom{\alpha}{2},  \dots \right)
  = \binom{k\alpha}{n-k},
\end{equation}
and
\begin{alignat}{3}\label{bakb}
  (e^t-1)^k  & = \sum_{n=0}^\infty \frac{k!}{n!}\stirb{n}{k} t^n & \qquad & \implies \qquad & \dm_{n,k}\left( \frac{1}{1!}, \frac{1}{2!}, \frac{1}{3!},  \dots \right) & = \frac{k!}{n!}\stirb{n}{k},
  \\
  (-\log(1-t))^k & = \sum_{n=0}^\infty \frac{k!}{n!}\stira{n}{k} t^n
 & \qquad & \implies \qquad &\dm_{n,k}\left( \frac{1}{1}, \frac{1}{2}, \frac{1}{3},  \dots \right) & =
  \frac{k!}{n!}\stira{n}{k}. \label{baka}
\end{alignat}
We may take the left identities of \e{bakb} and \e{baka} to be our definitions of the Stirling numbers, as in \cite[p. 51]{Comtet}, and develop all their properties from this starting point.
If we add or remove the first coefficient $a_1$ in $\dm_{n,k}$ there is a simple effect, by the binomial theorem:
\begin{lemma}
For $k \gqs 0$,
\begin{align}
  \dm_{n,k}(a_2,a_3, \dots) & = \sum_{j=0}^k (-a_1)^{k-j} \binom{k}{j} \dm_{n+j,j}(a_1,a_2, \dots), \label{add}\\
  \dm_{n,k}(a_1,a_1, \dots) & = \sum_{j=0}^k a_1^{k-j} \binom{k}{j} \dm_{n-k,j}(a_2,a_3, \dots). \label{drop}
\end{align}
\end{lemma}

In the remainder of this section some required identities are developed, allowing us to keep later proofs relatively self contained. Consult \cite[Chap. 5]{Comtet} and \cite[Sect. 6.1]{Knu} for further information on Stirling numbers.
Using \e{add} to remove the first coefficients in \e{bino}, \e{bakb} and \e{baka}  produces
\begin{align}\label{qwe}
  \dm_{n,k}\left( \binom{\alpha}{1}, \binom{\alpha}{2},\binom{\alpha}{3},  \dots \right)
  & =\sum_{j=0}^k (-1)^{k-j} \binom{k}{j} \binom{\alpha j}{n},\\
  \dm_{n,k}\left(  \frac{1}{2!}, \frac{1}{3!}, \frac{1}{4!}, \dots \right)
  & =\frac{k!}{(n+k)!} \sum_{j=0}^k (-1)^{k-j}  \binom{n+k}{n+j} \stirb{n+j}{j},\label{jf}\\
  \dm_{n,k}\left(  \frac{1}{2}, \frac{1}{3}, \frac{1}{4}, \dots \right)
  & =\frac{k!}{(n+k)!} \sum_{j=0}^k (-1)^{k-j}  \binom{n+k}{n+j} \stira{n+j}{j}.\label{jf2}
\end{align}
Another application of \e{add} shows
\begin{equation*}
  \dm_{n,k}\left( \frac{1}{1!}, \frac{1}{2!}, \frac{1}{3!},  \dots \right) =
  \sum_{j=0}^k (-1)^{k-j} \binom{k}{j}
   \dm_{n+j,j}\left( \frac{1}{0!}, \frac{1}{1!}, \frac{1}{2!},   \dots \right),
\end{equation*}
with the De Moivre polynomial on the right evaluating to $j^n/n!$. Hence by \e{bakb}, we obtain the well-known
\begin{equation} \label{ssub}
  \stirb{n}{k} = \frac 1{k!} \sum_{j=0}^k (-1)^{k-j} \binom{k}{j} j^n \qquad (n,k \gqs 0).
\end{equation}
 The right sides above may be expressed using the difference operator $\Delta$:
\begin{equation}\label{difo}
  \Delta f(t):= f(t+1)-f(t) \qquad \text{with} \qquad \Delta^k f(t) = \sum_{j=0}^k (-1)^{k-j} \binom{k}{j} f(t+j).
\end{equation}
Also $\Delta^k$ applied to $1/t$ gives a useful identity, which may be verified by induction, and we obtain
\begin{equation}\label{difi}
  \Delta^k \frac 1t = \frac{(-1)^k k!}{t(t+1) \cdots (t+k)} = \sum_{j=0}^k (-1)^{k-j} \binom{k}{j} \frac 1{t+j},
\end{equation}
as in \cite[Eqs. (5.40), (5.41)]{Knu}.
We next claim
\begin{equation}\label{difi2}
    \frac{t^k}{(1-t)(1-2t) \cdots (1-k t)} = \sum_{n=k}^\infty \stirb{n}{k} t^n.
\end{equation}
Using \e{difi} with $1/t$ for $t$, the left side of \e{difi2} equals
\begin{equation*}
  \frac{1}{(1/t-1)(1/t-2) \cdots (1/t-k)} = \frac{(-1)^k}{k!} \sum_{j=0}^k \binom{k}{j} \frac{(-1)^j}{1-j t}
  = \sum_{n=0}^\infty t^n \frac{1}{k!} \sum_{j=0}^k \binom{k}{j} (-1)^{k-j} j^n,
\end{equation*}
and the claim is proved with \e{ssub}.

\section{Symmetric polynomials} \label{sympo}
%\subsection{The polynomials $e_k$ and $h_k$}

A polynomial is called symmetric if it remains unchanged under any permutation of its variables.
The elementary symmetric polynomials $e_k$ naturally appear when the coefficients of a polynomial are expressed in terms of its roots, and so have a long history.
  In $n$ variables they may be defined as
\begin{equation}\label{elem}
  e_k = e_k(x_1,x_2, \dots, x_n) := \sum_{1\lqs j_1 < j_2< \cdots <j_k\lqs n} x_{j_1} x_{j_2} \cdots x_{j_k},
\end{equation}
where $e_0=1$ and $e_k=0$ for $k\gqs n+1$.
The  power-sum symmetric polynomials are
\begin{equation}\label{pow}
  p_k = p_k(x_1,x_2, \dots, x_n) := x_1^k+x_2^k+ \cdots +x_n^k.
\end{equation}
 Girard in 1625 expressed $p_k$ in terms of the elementary symmetric polynomials for $k=1,2,3,4$. For example, if $k=3$ then $p_3=e_1^3-3e_1 e_2+3e_3$ is true for all $n$, giving
$$
x_1^3+x_2^3+x_3^3 =\left(x_1+x_2+x_3 \right)^3 -3\left(x_1+x_2+x_3 \right)\left(x_1x_2+x_2x_3+x_1 x_3 \right)
+3\left(x_1 x_2 x_3 \right)
$$
when there are $n=3$ variables. Newton's identity of 1707  is
\begin{equation}\label{newto}
  k \cdot e_k =\sum_{j=1}^k (-1)^{j-1} p_j e_{k-j} \qquad (k\gqs 1),
\end{equation}
and Waring in 1762 used it to extend Girard's work to all $k$; see \cite{Funk,Gou}.
Waring's formula may be written conveniently with De Moivre polynomials as
\begin{equation}\label{war}
   p_k  = k\sum_{j=1}^{k} \frac{(-1)^{k-j}}{j} \dm_{k,j}\left(e_1, e_2, e_3,  \dots \right).
\end{equation}
 This gives a natural grouping  on the right since
   $\dm_{k,j}\left(e_1, e_2, e_3,  \cdots \right)$ has homogeneous degree $j$ in   $e_1$, $e_2, \dots$.
   % Saalsch\"{u}tz in \cite[\S 13]{Saa} describes \e{war} similarly.

The  complete homogeneous symmetric polynomials make a third family and they may be  defined as
\begin{equation}\label{hom}
  h_k = h_k(x_1,x_2, \dots, x_n) := \sum_{1\lqs j_1 \lqs j_2 \lqs  \cdots  \lqs j_k\lqs n} x_{j_1} x_{j_2} \cdots x_{j_k},
\end{equation}
with $h_0=1$. All three types form bases for the symmetric polynomials:

\begin{theorem} \label{bases}
If $f$ in $\Z[x_1, \dots,x_n]$ is symmetric then there exist unique  $\Phi_e, \Phi_h$ in $\Z[x_1, \dots,x_n]$ and $\Phi_p$ in $\Q[x_1, \dots,x_n]$ so that
$$f(x_1, \dots,x_n)=\Phi_e(e_1, \dots, e_n)=\Phi_h(h_1, \dots, h_n)=\Phi_p(p_1, \dots, p_n).$$
\end{theorem}

We are following \cite[Sect. I.2]{Macd}, which can be consulted for detailed explanations.
%See also \cite[Chap. 7]{Sta} and \cite[Chaps. 1-3]{Egg}.
The generating functions for $e_k$ and $h_k$ are
\begin{equation} \label{genf}
  E(t):=\sum_{k=0}^\infty e_k t^k = \prod_{j=1}^n (1+x_j t), \qquad
  H(t):=\sum_{k=0}^\infty h_k t^k = \prod_{j=1}^n \frac 1{1-x_j t},
  \end{equation}
and the power-sums generating function is
\begin{equation} \label{genf2}
  P(t):=\sum_{k=0}^\infty p_{k+1} t^k = \sum_{j=1}^n \frac{x_j}{1-x_j t}.
\end{equation}
%Note that the constant term of $P(t)$ is labelled $p_1$, and  $p_0$  is left undefined.
Differentiating the logs of  $E(t)$ and $H(t)$ shows the identities
\begin{equation}\label{ideh}
  P(-t)=\frac{E'(t)}{E(t)}, \qquad  P(t)=\frac{H'(t)}{H(t)}.
\end{equation}
In particular, $E'(t)=E(t)P(-t)$ implies  Newton's identity \e{newto}.

\begin{theorem}[Transition formulas] \label{tran}
For $r \gqs 0$, $\ell \gqs 1$,
\begin{align}\label{ehp}
  h_r  & = \sum_{k =0}^r (-1)^{r-k}\dm_{r,k}(e_1,e_2,e_3, \dots),
  &
  e_r  & = \sum_{k =0}^r (-1)^{r-k}\dm_{r,k}(h_1,h_2,h_3, \dots),\\
  e_r & = \sum_{k=0}^{r} \frac{(-1)^{r-k}}{k!} \dm_{r,k}\left(\frac{p_1}{1}, \frac{p_2}{2}, \frac{p_3}{3}, \dots \right), &
  \frac{p_\ell}\ell & = \sum_{k=1}^{\ell} \frac{(-1)^{\ell-k}}{k} \dm_{\ell,k}\left(e_1, e_2, e_3,  \dots \right), \label{ehp2}
  \\
  h_r & = \sum_{k=0}^{r} \frac{1}{k!} \dm_{r,k}\left(\frac{p_1}{1}, \frac{p_2}{2}, \frac{p_3}{3},  \dots \right), &
  \frac{p_\ell}\ell & = \sum_{k=1}^{\ell} \frac{(-1)^{k-1}}{k} \dm_{\ell,k}\left(h_1, h_2, h_3,  \dots \right).
  \label{ehp3}
\end{align}
\end{theorem}
\begin{proof}
Clearly $H(t)E(-t)=1$
and hence, by \e{pot} with $\alpha=-1$, we obtain \e{ehp}. Writing
\begin{equation} \label{dep2}
  \log E(-t) = \sum_{j=1}^n \log(1-x_j t) = -\sum_{j=1}^n \sum_{r=1}^\infty \frac{(x_j t)^r}{r}
  = - \sum_{r=1}^\infty \frac{p_r}{r}  t^r,
\end{equation}
and similarly for $\log H(t)$,  shows:
\begin{equation}\label{dep}
  E(-t) = \exp\left( - \sum_{r=1}^\infty \frac{p_r}{r}  t^r\right), \qquad \log H(t) = \sum_{r=1}^\infty \frac{p_r}{r}  t^r, \qquad
  H(t) = \exp\left(  \sum_{r=1}^\infty \frac{p_r}{r}  t^r\right).
\end{equation}
By \e{cox} and \e{log} we obtain
\e{ehp2} and \e{ehp3} from \e{dep2} and \e{dep}.
\end{proof}

%We may call the six identities in \e{ehp}, \e{ehp2} and \e{ehp3} {\em transition formulas}.
Theorem \ref{tran} has been known % in different forms
for a long time, with the right equality in \e{ehp2}  the Girard-Waring formula again. Expressing these transition formulas with  De Moivre polynomials has the advantage of providing a clear and uniform description. See \cite[Eq. (2.14)', Ex. 20]{Macd} for equivalent formulas as sums over partitions and also \cite[Thm. 2.2]{Sun05}.  With Proposition \ref{detp}  we may give determinant versions, as in \cite[Ex. 8]{Macd}. For example, using \e{far2} makes the left identity in \e{ehp3} into
\begin{equation*}
   h_n= \frac{1}{n!}\det
\left(
  \begin{matrix}
  p_1  & -1 & & & \\
  p_2  & p_1  & -2 & & \\
  p_3  & p_2  & p_1  & -3 & \\
  \vdots & & & \ddots &  \\
  p_n  & p_{n-1}  & \dots & p_2  & p_1
  \end{matrix}
  \right).
\end{equation*}
See \cite{Gou} for historical references to such determinant formulas, including Salmon's for $e_k$ and $p_k$  in 1876.

Note that the transition formulas are finite sums and independent of the $x_j$s and $n$. Using $\tilde{e}_k:=(-1)^{k-1}e_k$ instead of $e_k$ (and simplifying with \e{mulk}, \e{muln}) makes the signs appearing in Theorem \ref{tran}   more uniform; see also \e{rem2}.
By Theorem \ref{bases} we expect the polynomials expressing $p_\ell$ in terms of $e_k$ and $h_k$ on the right of \e{ehp2}, \e{ehp3} to have integral coefficients. This may be verified since the $\gcd$ of the coefficients of $\dm_{\ell,k}$ is $k/\gcd(\ell,k)$ by \cite[Thm. 4.1]{odm}.

Finally in this section we mention a family of symmetric polynomials  that were studied by MacMahon and  only include the terms of $h_k$ with $r$ distinct indices:
\begin{equation}\label{sss}
  S_{k,r}=S_{k,r}(x_1,x_2, \dots, x_n) := \sum_{\substack{ \alpha_1+ \alpha_2+\cdots + \alpha_r=k \\
  1\lqs j_1 < j_2 <  \cdots  < j_r\lqs n}} x^{\alpha_1}_{j_1} x^{\alpha_2}_{j_2} \cdots x^{\alpha_r}_{j_r}.
\end{equation}
They link $e_k$, $h_k$ and $p_k$ together with
\begin{equation*}
  e_k=S_{k,k}, \qquad h_k = \sum_{r=0}^k S_{k,r}, \qquad p_k=S_{k,1},
\end{equation*}
and based on his work in \cite[Vol. 1, p. 5]{Macm}, (see also \cite[Ex. 19]{Macd}, \cite[Lemma 2]{Hoff}), we have
\begin{equation}\label{mac}
  S_{k,r}= \sum_{j=r}^k (-1)^{j-r} \binom{j}{r} e_{j} h_{k-j}, \qquad \qquad
  e_{j} h_{k-j} = \sum_{r=j}^k  \binom{r}{j} S_{k,r}.
\end{equation}

\section{Symmetric triples} \label{trpss}

\subsection{Fundamental properties}
If the variables $x_1,x_2, \dots, x_n$ are assigned certain values, then the symmetric polynomials $e_k$, $h_k$ and $p_k$ also become numbers.  For a simple example, set each $x_j$ to $1$. Then by \e{genf}, \e{genf2} we easily find
\begin{equation*}
  E(t)=(1+t)^n, \qquad H(t)=(1-t)^{-n}, \qquad P(t)=\frac{n}{1-t},
\end{equation*}
so that
\begin{equation} \label{nck}
  e_k=\binom{n}{k}, \qquad h_k = \binom{n+k-1}{k}, \qquad p_k = n.
\end{equation}
These numbers are related by the transition formulas \e{ehp}, \e{ehp2}, \e{ehp3}, as well as the other identities we have seen, such as Newton's identity \e{newto}.  Focussing on the relationships between $e_k$, $h_k$ and $p_k$, we may discard  the original $x_j$s and $n$ and work directly with the symmetric triples in Definition \ref{mcd}. The following two propositions develop their basic properties and the proofs are straightforward exercises.

\begin{prop}\label{tri}
Let $(E(t),H(t),P(t))$ be a symmetric triple. Then we have
\begin{equation} \label{tub}
  H(t)E(-t)=1, \qquad E(t)P(-t)=E'(t), \qquad H(t)P(t)=H'(t),
\end{equation}
as well as the  variants
\begin{equation}\label{tub2}
  P(t)=E(-t)H'(t)=E'(-t)H(t).
\end{equation}
 The transition formulas of Theorem \ref{tran}
for their coefficients $e_k$, $h_k$ and $p_k$ all hold, and it follows that a symmetric triple is uniquely defined by specifying one of the formal series $E(t)$, $H(t)$ or $P(t)$. Also,
\begin{equation}\label{ones}
e_1=h_1=p_1.
\end{equation}
\end{prop}

\begin{prop}\label{group}
If $(E(t),H(t),P(t))$ is a symmetric triple then so are
\begin{align}
  (H(t),E(t),P(-t)), &  \label{macd1}\\
  (E(t)^\alpha,H(t)^\alpha,\alpha P(t)), & \qquad \text{for $\alpha$ any constant},\label{macd2}\\
  (E(-\phi(-t)),H(\phi(t)),\phi'(t) P(\phi(t))), & \qquad \text{for $\phi(t)$ any
  power series with $\phi(0)=0$}.\label{macd3}
\end{align}
Two symmetric triples $(E_1,H_1,P_1)$ and $(E_2,H_2,P_2)$ may be combined to get a third:
\begin{equation}
  (E_1 \cdot E_2,H_1 \cdot H_2,P_1+P_2). \label{macd4}
\end{equation}
%\begin{gather}
%  (E_1(t)E_2(t),H_1(t)H_2(t),P_1(t)+P_2(t)). \label{macd4}
%\end{gather}
\end{prop}

With  \e{macd2} for $\alpha=-1$ and \e{macd4}, these triples clearly form an abelian group with identity $(1,1,0)$. From the effect on $H(t)$ we will just call this group operation multiplication.
Proposition \ref{group} does not seem to have appeared before, though \e{macd2} is the first part of \cite[Thm. 2.4]{Sun05}.

%\begin{proof}[Discussion] \end{proof}
\begin{eg}[Binomial coefficients] \label{ex-bc}
\begin{equation}\label{weec}
  \left((1+t)^\alpha, \ (1-t)^{-\alpha}, \ \frac{\alpha}{1-t}\right) \qquad \text{with} \qquad e_k=\binom{\alpha}{k}, \ h_k=\binom{k+\alpha-1}{k}, \ p_k= \alpha.
\end{equation}
\end{eg}
\begin{proof}[Discussion]
This is \cite[Ex. 1]{Macd}. Here $\alpha$ is an arbitrary complex number or variable and it is easy to verify that the conditions of Definition \ref{mcd} are satisfied. Each of the equalities in \e{tub}, \e{tub2} lead to convolution identities similar to Newton's identity \e{newto}, which in this case gives (as in \cite[Eq. (5.16)]{Knu})
\begin{equation*}
  k\binom{\alpha}{k} = (-1)^{k-1}\sum_{j=0}^{k-1} (-1)^{j} \alpha\binom{\alpha}{j}.
\end{equation*}
The transition formula expressing $p_k$ in terms of $e_k$ (Girard-Waring) yields
\begin{equation} \label{dart}
  \frac{\alpha}{r}=\sum_{k=1}^r \frac{(-1)^{k+r}}{k} \dm_{r,k}\left( \binom{\alpha}{1}, \binom{\alpha}{2},\binom{\alpha}{3},  \dots \right).
\end{equation}
It follows from \e{qwe}, after some simplifications,  that  \e{dart} is equivalent to
\begin{equation*}
  \sum_{k=0}^r (-1)^{k}  \binom{r}{k} \binom{\alpha k-1}{r-1} = 0.
\end{equation*}
This can be seen directly using \e{difo}. Since each difference lowers the degree, the $r$th difference of the degree $r-1$ polynomial $\binom{\alpha x-1}{r-1}$ in $x$ must be zero.
  The transition formula for $e_k$ in terms of $p_k$ also shows
\begin{align}
  \binom{\alpha}{r} & = \sum_{k=0}^{r} \frac{(-1)^{r+k}}{k!} \dm_{r,k}\left(\frac{\alpha}{1}, \frac{\alpha}{2}, \frac{\alpha}{3}, \cdots \right) \notag\\
  & = \sum_{k=0}^{r} (-1)^{r+k}\frac{\alpha^k}{k!} \dm_{r,k}\left(\frac{1}{1}, \frac{1}{2}, \frac{1}{3}, \cdots \right)
  = \frac{1}{r!} \sum_{k=0}^{r} (-1)^{r-k} \stira{r}{k} \alpha^{k}, \label{dog}
\end{align}
by \e{mulk} and  \e{baka}, giving the well-known expansion of the falling factorial as in \cite[Eq. (6.13)]{Knu}. Equivalently, for the rising factorial,
\begin{equation}\label{rise}
  \alpha(\alpha+1) \cdots (\alpha+r-1) = \sum_{k=0}^{r} \stira{r}{k} \alpha^{k}.
\end{equation}
MacMahon's polynomial \e{sss} takes the elegant form
\begin{equation*}
  S_{k,r} = \binom{k-1}{r-1} \binom{\alpha}{r} \qquad (k\gqs 1),
\end{equation*}
after simplification.
See also \cite[Sect. 3.2]{Egg} for further identities in this binomial case.
\end{proof}

As an example of the manipulations that are possible with Proposition \ref{group}, consider \e{weec} with $\alpha=1/2$. Also apply \e{macd1} and we obtain the two symmetric triples
\begin{equation}\label{twos}
  \left((1+t)^{1/2}, \ (1-t)^{-1/2}, \ \frac{1}{2(1-t)}\right), \qquad  \left((1-t)^{-1/2}, \ (1+t)^{1/2},  \ \frac{1}{2(1+t)}\right).
\end{equation}
Multiplying them produces
\begin{equation}\label{h1}
  (E(t),H(t),P(t)):=  \left(\left(\frac{1+t}{1-t}\right)^{1/2}, \ \left(\frac{1+t}{1-t}\right)^{1/2},  \ \frac{1}{1-t^2}\right).
\end{equation}
Dividing the first by the second in \e{twos} produces
\begin{equation}\label{h2}
  (E^\dag(t),H^\dag(t),P^\dag(t)):=  \left((1+t)^{1/2}(1-t)^{1/2}, \  (1+t)^{-1/2}(1-t)^{-1/2},  \ \frac{t}{1-t^2}\right).
\end{equation}
Then we easily find $p_{2n-1}=1$, $p_{2n}=0$ and $p^\dag_{2n-1}=0$, $p^\dag_{2n}=1$. Also
$$
H^\dag(t)=(1-t^2)^{-1/2} \qquad \text{implies} \qquad h^\dag_{2n}=(-1)^n \binom{-1/2}{n}=\binom{n-1/2}{n}=\frac 1{2^{2n}}\binom{2n}{n},
$$
as in \cite[(5.36), (5.37)]{Knu}, with $h^\dag_{2n-1}=0$. Similarly $h_{2n-1}=h_{2n}=2^{-2n}\binom{2n}{n}$. By the transition formula between $h_k$ and $p_k$ we obtain the following identities, equivalent to those  given in \cite[Ex. 15]{Macd}:
\begin{equation}\label{macoe}
  \sum_{k=0}^{2n} \frac{1}{k!} \dm_{2n,k}\left(\frac 1{1}, 0, \frac 1{3}, 0, \frac 1{5},0, \dots \right)  =
  \sum_{k=0}^{2n} \frac{1}{k!} \dm_{2n,k}\left(0,\frac 1{2}, 0, \frac 1{4}, 0, \frac 1{6},\dots \right)  = \frac 1{2^{2n}}\binom{2n}{n}.
\end{equation}

Symmetric triples correspond to homomorphisms $\phi$ from the ring $\Lambda$ of symmetric functions to a simpler ring, such as $\Q[\alpha]$ in Example \ref{ex-bc}. As described in \cite[Sect. I.2]{Macd} and \cite[Chap. 2]{MR15},
symmetric functions  are formal power series in infinitely many variables $x_1, x_2, \dots$ with coefficients in $\Z$ (or sometimes $\Q$). They have  bounded degree and are unchanged under any finite permutation of these variables. Symmetric polynomials appear as the special case  when $x_{n+1}$, $x_{n+2}, \dots$ are all set to $0$. The elements  $e_k$, $h_k$ and $p_k$  defined by \e{elem}, \e{pow} and \e{hom}  with $n=\infty$ each form a basis of $\Lambda$, as in
Theorem \ref{bases}, so that $\phi$ may be defined uniquely by specifying it on one of these bases. The homomorphism $\phi$ is called a specialization, see \cite[Sect. 7.8]{Sta}, and the coefficients  of the series in a symmetric triple $(E(t), H(t), P(t))$ correspond to $\phi(e_k)$, $\phi(h_k)$ and $\phi(p_k)$ in this picture.
For each of our triple examples, the images of the other bases of $\Lambda$, such as the Schur functions, can be considered. See for example \cite[Ex. 23, p. 58]{Macd}.

%In this paper, $e_k$, $h_k$ and $p_k$ will be given as explicit  numbers, polynomials or rational functions.
%A symmetric triple $(E(t),H(t),P(t))$ does not refer to values of the original variables $x_j$.
If $E(t)$  is a polynomial of degree $n$, then  $x_1$, $x_2, \dots, x_n$ may be recovered as the negative reciprocals of the zeros of $E(t)$, as in \e{genf}. If $E(t)$ is not a polynomial then it can be instructive to look for zeros of $E(t)$, or equivalently poles of $H(t)$. However,  $E(t)$ is a formal series and not necessarily a well-defined function. % of $t$.

\subsection{Changing variables in a triple} \label{calc}
The following properties are based on \e{macd3} and will be useful, especially in section \ref{etype}.
If $(E,H,P)$ is a symmetric triple then the change of variable $t \to ct$ gives the new triple
\begin{equation} \label{ccc}
  (\tilde{E},\tilde{H},\tilde{P})= \left( E(ct), H(ct), c P(ct) \right) \qquad \text{with} \qquad
  \tilde{e}_k = c^k e_{k}, \quad \tilde{h}_k =  c^k h_{k}, \quad \tilde{p}_k =  c^k p_{k}.
\end{equation}
 Similarly, $t \to t^m$ for a positive integer $m$ produces
\begin{equation*}
   (\hat{E},\hat{H},\hat{P})= \left( E((-1)^{m-1} t^m),  \  H(t^m),  \  m \cdot t^{m-1} P(t^m) \right)
\end{equation*}
with new coefficients that are $0$ unless $m$ divides the index, in which case
\begin{equation*}
  \hat{e}_{mk} =
  \begin{cases}
  (-1)^k e_{k} & \text{ if $m$ is even};\\
  e_{k} & \text{ if $m$ is odd, }
  \end{cases}
   \qquad \hat{h}_{mk} =   h_{k}, \qquad \hat{p}_{mk} =  m \cdot p_{k}.
\end{equation*}
This is reversed in the next easily verified result.

\begin{prop} \label{root}
Let $m$ be a positive integer and $(E,H,P)$  a symmetric triple. Suppose the coefficients $h_k$ are zero whenever $k$ is not a multiple of $m$. Then we can construct the new symmetric triple
\begin{equation*}
   (\hat{E},\hat{H},\hat{P})= \left( E((-1)^{(m-1)/m} t^{1/m}),  \  H(t^{1/m}),  \  \frac{t^{1/m-1}}m P(t^{1/m}) \right)
\end{equation*}
with  coefficients
\begin{equation*}
  \hat{e}_{k} =
  \begin{cases}
  (-1)^k e_{mk} & \text{ if $m$ is even};\\
  e_{mk} & \text{ if $m$ is odd, }
  \end{cases}
   \qquad \hat{h}_{k} =   h_{mk}, \qquad \hat{p}_{k} =  \frac{p_{mk}}m.
\end{equation*}
\end{prop}

Theorem 2.5 in \cite{Sun05} uses roots of unity to make pairs $(H,P)$ where $H(t)$ only contains powers that are multiples of $m$. We give another version of this result next. Recall that, as usual, $R$ is an integral domain containing $\Q$.

\begin{theorem}
Let $m$ be a positive integer and $(E,H,P)$  a symmetric triple with coefficients in $R$.  Then, for $\omega=e^{2\pi i/m}$, we can make the new symmetric triple
\begin{equation} \label{newtt}
   (\tilde{E},\tilde{H},\tilde{P})= \left( \prod_{j=0}^{m-1} E(\omega^j t),  \  \prod_{j=0}^{m-1} H(\omega^j t), \ \sum_{j=0}^{m-1} \omega^j P(\omega^j t)\right)
\end{equation}
 with coefficients in $R[\omega]$. These new coefficients $\tilde{e}_k$, $\tilde{h}_k$ and $\tilde{p}_k$ are all zero unless $m\mid k$ and in fact are all in $R$. We have $\tilde{p}_k = m \cdot p_k$ when $m\mid k$.
\end{theorem}
\begin{proof}
Use \e{ccc} and the group operation \e{macd4} to check  that \e{newtt} is a symmetric triple. Since $\tilde{E}(\omega t) = \tilde{E}(t)$, it follows that $\tilde{e}_k \omega^k = \tilde{e}_k$ and hence that $\tilde{e}_k$ must be $0$ unless $m | k$. Similarly for $\tilde{h}_k$ and $\tilde{p}_k$. Also
\begin{equation*}
  \sum_{k=1}^\infty \tilde{p}_k t^{k-1} = \sum_{j=0}^{m-1} \omega^j \sum_{r=1}^\infty p_r (\omega^j t)^{r-1}
  =  \sum_{r=1}^\infty p_r t^{r-1} \sum_{j=0}^{m-1} \omega^{r j}
\end{equation*}
and hence $\tilde{p}_k = m \cdot p_k$ when $m\mid k$. In particular, $\tilde{p}_k \in R$ and the transition formulas now show that $\tilde{e}_k$ and $\tilde{h}_k$ are also in $R$.
\end{proof}

\section{Symmetric triple examples}
A simple way to build a symmetric triple is to start with any formal power series that has constant term $1$. Letting this be $H(t)$, we produce
\begin{equation}\label{hhh}
  \left(\frac{1}{H(-t)}, \ H(t), \ \frac{H'(t)}{H(t)}\right),
\end{equation}
and  this lets us study the reciprocal and the derivative of $H(t)$. We will also include  compositional inverses in section \ref{cinv}.

\subsection{Examples with equal coefficients}

\begin{prop}
We have $e_k=h_k$ for all $k\gqs 1$ if and only if $P(t)$ is even.
\end{prop}
\begin{proof}
Recall that $e_0=h_0=1$. Then
\begin{align*}
  E(t)=H(t) & \iff \log E(t) = \log H(t) \\
  & \iff \frac{d}{dt}\log E(t) = \frac{d}{dt}\log H(t) \iff P(-t)=P(t). \qedhere
\end{align*}
\end{proof}

\begin{prop} \label{h=p}
Suppose $h_k=p_k$ for all $k\gqs 1$. Then for some constant $c$,
\begin{equation*}
  (E(t),H(t),P(t))=\left(1+ct,\ \frac{1}{1-ct}, \ \frac{c}{1-ct} \right).
\end{equation*}
\end{prop}
\begin{proof}
We have
\begin{equation*}
  H(t)=1+t P(t) \iff \frac 1{E(-t)}=1+ t \frac{E'(-t)}{E(-t)} \iff 1=E(t)-t E'(t).
\end{equation*}
Differentiating shows $E''(t)=0$ and hence $E(t)=1+c t$.
\end{proof}

There is a similar result if $h_k=\alpha \cdot p_k$ for all $k\gqs 1$ and any fixed $\alpha$.
The next case is stated in \cite[Ex. 16]{Macd}. The Bernoulli numbers are defined with $B_k:=k![t^k]t/(e^t-1)$.

\begin{prop}
Suppose $e_k=p_k$ for all $k\gqs 1$. Then for some constant $c$,
\begin{equation*}
  (E(t),H(t),P(t))=\left(\frac{-ct}{e^{-ct}-1},  \  \frac{e^{ct}-1}{ct},  \   \frac{1}{t}\left( -1+\frac{-ct}{e^{-ct}-1}\right) \right),
\end{equation*}
with
\begin{equation} \label{eqq}
  e_k = (-1)^k \frac{c^k B_k}{k!}, \qquad h_k= \frac{c^k}{(k+1)!}, \qquad p_k = (-1)^k \frac{c^k B_k}{k!}.
\end{equation}
\end{prop}
\begin{proof}
Expressing $E(t)=1+t P(t)$ in terms of $H$ implies $H(-t) \frac{d}{dt}(t H(t))=H(t)$. Set $F(t):=t H(t)$ and we find
\begin{equation}\label{ff}
  F(-t) F'(t)=-F(t).
\end{equation}
This implies the symmetry
\begin{equation}\label{ff2}
  F'(-t) = -\frac{F(-t)}{F(t)} = \frac{1}{F'(t)}.
\end{equation}
Differentiating \e{ff} and simplifying with \e{ff2} shows
\begin{equation*}
  1+F(t) F''(t)/F'(t) = -F'(t).
\end{equation*}
Differentiating this once more finds $\frac{d}{dt}(F''(t)/F'(t))=0$. Therefore
\begin{equation*}
  \frac{d^2}{dt^2} \log F'(t) = 0 \implies \log F'(t) = ct+b  \implies  F'(t) = \exp(ct+b)
\end{equation*}
for some constants $c$ and $b$. But $F'(t)=H(t)+t H'(t)$, making $F'(0)=1$ and so $b=0$. Next, for some $a$,
\begin{equation*}
  F(t) = \exp(ct)/c+a.
\end{equation*}
Finally, $F(0)=0$ means $a=-1/c$ and the result follows.
\end{proof}

% \begin{proof}[Discussion] \end{proof}
\subsection{Stirling and harmonic numbers}

\begin{eg}[Stirling numbers, power sums] \label{xstir} For  $n \gqs 1$,
\begin{equation}\label{xj2}
  e_k = \stira{n+1}{n+1-k}, \qquad h_k = \stirb{n+k}{n},  \qquad p_k = 1^k+2^k+\cdots +n^k. %\mathcal{S}_k(n):=
\end{equation}
\end{eg}
\begin{proof}[Discussion]
This example comes from setting $x_j=j$ in \e{genf}, \e{genf2} to get  the generating functions
\begin{equation}\label{xj}
  E(t)=\prod_{j=1}^n (1+jt), \qquad H(t)=\prod_{j=1}^n \frac 1{1-jt}, \qquad P(t)=\sum_{j=1}^n \frac j{1-jt}.
\end{equation}
Their  coefficients in \e{xj2} follow from \e{rise} with $\alpha=1/t$ and \e{difi2}. If $k\gqs n+1$ then $e_k$ must be $0$ and we may extend the definition of $\stira{n}{k}$ to be $0$ when $k\lqs 0$.
We now have many identities connecting these numbers. For example $P(t)=E'(-t)H(t)$ from \e{tub2} implies
\begin{equation} \label{peh}
  p_k=\sum_{j=1}^k (-1)^{j-1} j e_j h_{k-j}
\end{equation}
so that, for $n,k \gqs 1$,
\begin{equation*}
  1^k+2^k+\cdots +n^k = \sum_{j=1}^k (-1)^{j-1} j \stira{n+1}{n+1-j} \stirb{n+k-j}{n},
\end{equation*}
as in \cite[Prop. 3.17]{Egg}, (including the missing $j$ factor).
The transition formula $p_k \leftrightarrow h_k$ also shows that
\begin{equation*}
  1^k+2^k+\cdots +n^k = \sum_{j=1}^{k} (-1)^{j-1}\frac{k}{j} \dm_{k,j}\left(\stirb{n+1}{n}, \stirb{n+2}{n}, \cdots \right).
\end{equation*}
By \e{ones} we can see
\begin{equation*}
  \stira{n+1}{n} =  \stirb{n+1}{n} = \binom{n+1}{2}.
\end{equation*}
%See also \cite[Sect. 3.3]{Egg} for further identities in this Stirling number case.
 In 1796 Kramp found formulas for the symmetric polynomials $e_k$ and $h_k$ in \e{xj2}. This is discussed by Knuth \cite[p. 413]{K2}, and in our notation they give, for $m,k\gqs 0$,
\begin{align}\label{kr}
  \stirb{m+k}{k} & =  m! \binom{m+k}{m} \sum_{j=0}^m \binom{k}{j} \dm_{m,j}\left(\frac{1}{2!},\frac{1}{3!},\frac{1}{4!}, \dots \right), \\
  \stira{m+k}{k} & = m! \binom{m+k}{m} \sum_{j=0}^m \binom{k}{j} \dm_{m,j}\left(\frac{1}{2}, \frac{1}{3}, \frac{1}{4}, \dots \right). \label{kr2}
\end{align}
We derive them after \e{nut}. Recall the simpler identities \e{bakb}, \e{baka}.
\end{proof}

In general, as described in \cite{Kon}, there is a natural combinatorial interpretation of $e_k$ and $h_k$ when  $x_1, x_2, \dots, x_n$ are nonnegative integers. We have this situation in \e{nck} with $x_j=1$, Example \ref{xstir} with $x_j=j$ and Example \ref{qbc} with $x_j=q^{j-1}$ if $q\in \Z_{\gqs 2}$. Suppose there are $n$ boxes with $x_j$ different balls in the  box with index $j$. Then $e_k=e_k(x_1,\dots,x_n)$ is the number of ways to choose $k$ balls from these boxes, where we mean one ball from each of $k$ different boxes and the order of the boxes is unimportant. The number $h_k=h_k(x_1,\dots,x_n)$ gives a similar count, but the boxes do not have to be different. We may extend this to $p_k=p_k(x_1,\dots,x_n)$, giving the number of ways to choose $k$  balls, with replacement, from a single box.
 Also in this setup, $\dm_{n,k}(x_1,x_2, \dots)$ counts the number of ways to choose one  ball from each of $k$ (not necessarily different) boxes with indices summing to $n$, and where the order of the boxes is important.

\begin{prop}[The generalized Pascal identities]
For $n,k \gqs 2$ and arbitrary $x_1, \dots, x_n$,
\begin{align}
  e_{k}(x_1,\dots,x_{n}) & = e_{k}(x_1,\dots,x_{n-1}) + x_n \cdot e_{k-1}(x_1,\dots,x_{n-1}), \label{epas}\\
  h_{k}(x_1,\dots,x_{n}) & = h_{k}(x_1,\dots,x_{n-1}) + x_n \cdot h_{k-1}(x_1,\dots,x_{n}), \label{hpas}\\
  p_{k}(x_1,\dots,x_{n}) & = p_{k}(x_1,\dots,x_{n-1}) + x_n \cdot p_{k-1}(0,\dots,0,x_{n}). \label{ppas}
\end{align}
\end{prop}

Equation \e{ppas} is clear and \e{epas}, \e{hpas} are easily demonstrated using the generating functions \e{genf}.
In the case of the Stirling numbers in Example \ref{xstir}, \e{epas} and \e{hpas} show that, for $n\gqs 3$, $k\gqs 2$,
\begin{equation}\label{s2}
  \stira{n+1}{k}=n \stira{n}{k} + \stira{n}{k-1}, \qquad  \stirb{n+1}{k}=k \stira{n}{k} + \stira{n}{k-1}.
\end{equation}
These relations allow the Stirling numbers to be extended to all integers $n$ and $k$. In fact
\e{s2} combined with the initial conditions $\stira{n}{0} =\stirb{n}{0} =\delta_{n,0}$  and $\stira{0}{k} =\stirb{0}{k} =\delta_{k,0}$ gives an elegant alternate definition for them. The remarkable duality $\stira{n}{k} = \stirb{-k}{-n}$ can also be observed; see \cite[pp. 266, 267]{Knu} and \cite{K2}.

\begin{eg}[Harmonic numbers] \label{onic} For  $n \gqs 1$,
\begin{equation}\label{xhar}
  e_k = \stira{n+1}{k+1}/n!, \qquad h_k = \stirc{n}{k},  \qquad p_k = H_n^{(k)}.
\end{equation}
\end{eg}

\begin{proof}[Discussion]
 Recall the definition \e{hnk} for $H_n^{(k)}$.  Here we are setting $x_j=1/j$ for $1\lqs j\lqs n$. The new notation $\stirc{n}{k}$ is meant to suggest  multiset coefficients, as in $\left( \binom{n}{k}\right):=\binom{n+k-1}{k}$, but of an unusual harmonic kind. By analogy with the above discussion, $\stirc{n}{k}$ represents the number of ways to choose $k$ balls, with replacement, from $n$ boxes where the $j$th box contains $1/j$ balls. The generating functions are
\begin{equation}\label{xha}
  E(t)=\frac{1}{n!}\prod_{j=1}^n (j+t), \qquad H(t)=n! \prod_{j=1}^n \frac 1{j-t}, \qquad P(t)=\sum_{j=1}^n \frac 1{j-t},
\end{equation}
with the formula for $e_k$ in \e{xhar} coming from \e{rise}.

\begin{prop} \label{flax}
We have
\begin{equation}\label{new}
  \stirc{n}{k} = \sum_{j=1}^n (-1)^{j-1} \binom{n}{j}  \frac{1}{j^k} \qquad (n,k \in \Z_{\gqs 1}).
\end{equation}
\end{prop}
\begin{proof}
By \e{difi},
\begin{equation*}
  \frac{H(t)}{-t} =\sum_{j=0}^n \binom{n}{j} \frac{(-1)^j}{j-t}
  =-\frac{1}{t}+ \sum_{j=1}^n \frac{(-1)^j}j \binom{n}{j} \sum_{r=0}^\infty \frac{t^r}{j^r},
\end{equation*}
and rearranging and comparing powers of $t$  gives \e{new}.
\end{proof}

See \cite[Thm. 2.1]{Bat} for another proof and a discussion of the history of this result. By \e{hpas}, for $n,k \gqs 2$, we have (after multiplying through by $n$) the Pascal identity
\begin{equation}\label{harm2}
 n\stirc{n}{k}=n\stirc{n-1}{k}+  \stirc{n}{k-1}.
\end{equation}

\begin{theorem} \label{you}
Suppose we define the harmonic multiset numbers $\stirc{n}{k}$ using the recursion \e{harm2} and the initial conditions
\begin{equation}\label{hc}
  \stirc{n}{-1}=\delta_{n,1}, \qquad \stirc{-1}{k}=\delta_{k,1}.
\end{equation}
Then $\stirc{n}{k}$ is well-defined for all $n,k \in \Z$ and agrees with our previous definition. The Stirling numbers are special cases:
\begin{equation}\label{hc2}
  \stirb{n}{k}=\frac{(-1)^{k+1}}{k!} \stirc{k}{-n}, \qquad \stira{n}{k}= (n-1)! \stirc{-n}{k} \qquad (n \gqs 1, k \gqs 0).
\end{equation}
\end{theorem}
\begin{proof}
For clarity, let $S(n,k)$ denote the  numbers defined recursively by \e{harm2} and \e{hc}. We first note that \e{harm2} implies  $S(0,k)=0$ for all $k$ and this does not conflict with \e{hc}. The recursion then gives $S(1,k)=1$ for all $k$, and next that $S(n,0)=1$ for all $n \gqs 1$.

Let $T(n,k)$ denote the right side of \e{new}. Check that $T(1,k)=T(n,0)=1$ for $n\gqs 1$ and all $k$, agreeing with $S(n,k)$. It follows from  the usual Pascal identity that $T(n,k)$ satisfies the same recursive relation  as $S(n,k)$ for $n\gqs 2$. Hence $S(n,k)=T(n,k)$ for all $n,k$ with $n\gqs 1$. That means $S(n,k)=\stirc{n}{k}$ for all $n,k \gqs 1$ and the new definition agrees with the old one. The formula \e{ssub}
 now shows the left identity in \e{hc2} is true.

%\vskip 10cm

%*************************************
% stirling

\SpecialCoor
\psset{griddots=5,subgriddiv=0,gridlabels=0pt}
\psset{xunit=0.7cm, yunit=0.5cm}
\psset{linewidth=1pt}
\psset{dotsize=2pt 0,dotstyle=*,arrowsize=2pt 3}

\newrgbcolor{pale}{0.9 0.9 1}
\newrgbcolor{pale2}{0.8 0.8 1}
\newrgbcolor{blu2}{0.4 0.4 1}

\begin{figure}[h]
\begin{center}
\begin{pspicture}(-10,-7)(10,7) %\psgrid

\psline[linecolor=gray]{->}(-6,-7)(-6,7)
\multirput(-6,-6)(0,1){13}{\psline[linecolor=gray](0,0)(-0.15,0)}

\psline[linecolor=gray]{->}(-7,-6)(7,-6)
\multirput(-6,-6)(1,0){13}{\psline[linecolor=gray](0,0)(0,-0.2)}

\rput(-6.7,6.5){$n$}

\rput(-6.7,-3.5){$_{-4}$}
\rput(-6.7,-1.5){$_{-2}$}
\rput(-6.5,0.5){$_{0}$}
\rput(-6.5,2.5){$_{2}$}
\rput(-6.5,4.5){$_{4}$}

\rput(6.5,-6.7){$k$}

\rput(0.5,-6.5){$_{0}$}
\rput(2.5,-6.5){$_{2}$}
\rput(4.5,-6.5){$_{4}$}
\rput(-1.5,-6.5){$_{-2}$}
\rput(-3.5,-6.5){$_{-4}$}

\pspolygon[linecolor=pale,fillstyle=solid,fillcolor=pale]
(-5,1)(0,1)(0,2)(-1,2)(-1,3)(-2,3)(-2,4)(-3,4)(-3,5)(-4,5)(-4,6)(-5,6)
%\psline[linecolor=blue]
%(-5,1)(0,1)(0,2)(-1,2)(-1,3)(-2,3)(-2,4)(-3,4)(-3,5)(-4,5)(-4,6)(-5,6)

\pspolygon[linecolor=pale,fillstyle=solid,fillcolor=pale]
(0,1)(6,1)(6,6)(0,6)

\pspolygon[linecolor=pale,fillstyle=solid,fillcolor=pale]
(1,-5)(1,0)(2,0)(2,-1)(3,-1)(3,-2)(4,-2)(4,-3)(5,-3)(5,-4)(6,-4)(6,-5)

\pspolygon[linecolor=pale2,fillstyle=solid,fillcolor=pale2]
(-5,1)(6,1)(6,2)(-5,2)

\pspolygon[linecolor=pale2,fillstyle=solid,fillcolor=pale2]
(0,6)(0,1)(1,1)(1,6)

\pspolygon[linecolor=pale2,fillstyle=solid,fillcolor=pale2]
(1,-5)(1,0)(2,0)(2,-5)

\pspolygon[linecolor=blu2,fillstyle=solid,fillcolor=blu2]
(1,-1)(1,0)(2,0)(2,-1)

\pspolygon[linecolor=blu2,fillstyle=solid,fillcolor=blu2]
(-1,1)(0,1)(0,2)(-1,2)

\psline[linecolor=orange](0,-5)(0,6)
\psline[linecolor=orange](1,-5)(1,6)
\psline[linecolor=orange](-5,0)(6,0)
\psline[linecolor=orange](-5,1)(6,1)

\rput(3.5,4){$\displaystyle \stirc{n}{k}$}

\rput(-9,4){$\displaystyle (-1)^{n+1} n! \stirb{-k}{n}$}

\pscurve{->}(-6.75,3.8)(-5,3.8)(-3.6,3.5)

\rput(8.7,-3){$\displaystyle \frac 1{(-n-1)!} \stira{-n}{k}$}
\pscurve{->}(6.5,-3.1)(5.5,-3.1)(4,-3.5)

\end{pspicture}
\caption{The harmonic multiset numbers $\stirc{n}{k}$ for $-5 \lqs n,k \lqs 5$ \label{fig}}
\end{center}
\end{figure}

%************************************

The recursion \e{harm2} gives no information about $S(-1,k)$, so we may impose the initial condition $S(-1,k)=\delta_{k,1}$ from \e{hc}. Then recursively we obtain $S(n,k)=0$ for all $n,k \lqs 0$, (and there is no conflict with $S(n,-1)=\delta_{n,1}$). Now set $U(n,k)$ to be $(n-1)! S(-n,k)$ for $n\gqs 1, k \gqs 0$. We have $U(n,0)=0=\stira{n}{0}$ for $n\gqs 1$ and $U(1,k)=\delta_{k,1}=\stira{1}{k}$ for $k \gqs 0$. Also $U(n,k)$ satisfies the same recursion as $\stira{n}{k}$ in \e{s2}. The right identity in \e{hc2} follows.
\end{proof}

The initial conditions \e{hc} at $n=\pm 1$, $k=\mp 1$ are highlighted in Figure \ref{fig}.
They are the natural seeds to obtain the Stirling numbers in quadrants 2 and 4 since the  recursions satisfied by the Stirling numbers and the harmonic multiset numbers are  essentially just rotated versions of  the same relation.  The reason that quadrant 1 is full of positive numbers instead of zeros, (compare with  the top right of \cite[Table 267]{Knu} in the usual Stirling case), is the slight difference in \e{harm2} at $n=0$ from the Stirling relations, and the fact that $\stirc{0}{0}$ is $0$ instead of $1$.

With our extended definition of $\stirc{n}{k}$, the symmetric triple \e{xhar} becomes
\begin{equation}\label{xxhar}
  e_k = \stirc{-n-1}{k+1}, \qquad h_k = \stirc{n}{k},  \qquad p_k = H_n^{(k)}.
\end{equation}
Also by \e{ones},
\begin{equation}\label{harm1}
 \stira{n+1}{2}/n!=\stirc{-n-1}{2} = \stirc{n}{1}=H_n \qquad (n\gqs 1).
\end{equation}
With the proof of Theorem \ref{you} we see that \e{new} is valid for all $n\gqs 0$ and $k\in \Z$, so that
\begin{equation}\label{nw2}
  \stirc{0}{k} = 0, \qquad \stirc{1}{k} = 1, \qquad \stirc{2}{k} = 2-\frac 1{2^k}, \qquad \stirc{3}{k} = 3-\frac 3{2^k}+\frac 1{3^k} \qquad (k\in \Z).
\end{equation}
With \e{xhar} we may express $\stirc{n}{k}$ for $n,k\gqs 1$ in terms of the Stirling cycle numbers or the higher order  harmonic numbers using the transition formulas. This last was also done in \cite[p. 7]{FS95}, \cite[Sect. 3]{Bat} and \cite[Sect. 3.1]{Roman}. We obtain
\begin{equation}\label{chd}
  \stirc{n}{k} = \sum_{j=0}^k \frac 1{j!} \dm_{k,j}\left(\frac{H_n^{(1)}}{1},\frac{H_n^{(2)}}{2}, \frac{H_n^{(3)}}{3}, \cdots \right) \qquad (n\gqs 1, k\gqs 0),
\end{equation}
with, for $n\gqs 1$,
\begin{equation}\label{chd2}
  \stirc{n}{0}=1, \qquad \stirc{n}{1}=H_n, \qquad \stirc{n}{2}=\frac{H_n^2}2+\frac{H_n^{(2)}}2, \qquad
  \stirc{n}{3}=\frac{H_n^3}6+\frac{H_n H_n^{(2)}}2 +\frac{H_n^{(3)}}3.
\end{equation}

We also find
\begin{align} \label{bg}
  H_n^{(k)} & = k \sum_{j=1}^k \frac{(-1)^{k-j}}{j \cdot (n!)^j} \dm_{k,j}\left( \stira{n+1}{2},\stira{n+1}{3} , \cdots \right),\\
  \stira{n+1}{k+1} & = n! \sum_{j=0}^k \frac {(-1)^{k-j}}{j!} \dm_{k,j}\left(\frac{H_n^{(1)}}{1},\frac{H_n^{(2)}}{2}, \frac{H_n^{(3)}}{3}, \cdots \right), \label{bg3}
\end{align}
with, for instance,
\begin{equation} \label{bg4}
\stira{n+1}{3}= \frac{n!}{2}\left( H_n^2-H_n^{(2)}\right), \qquad \stira{n+1}{4}= \frac{n!}{6}\left( H_n^3-3H_n H_n^{(2)}+2H_n^{(3)}\right).
\end{equation}
For further interesting properties of $\stirc{n}{k}$ when $n,k \gqs 1$, including connections to polylogarithms,  see \cite{Bat} where they are denoted $S_n(k)$ and \cite{Roman} where they are called {\em Roman harmonic numbers} with notation $c_n^{(k)}$, based on earlier work by Loeb, Rota and Roman. Up to a sign and factorial, they are also the {\em negative-positive Stirling numbers} in \cite{Bra}. The case $k<0$ is briefly mentioned in \cite{Bat} with the connection to the Stirling subset numbers $\stirb{n}{k}$ noted; $n<0$ is not discussed. The   negative $k$ case is not considered in \cite{Roman}, though $c_n^{(k)}$ with negative $n$ is studied and the Stirling cycle numbers $\stira{n}{k}$ are found. The initial conditions and recursion relation for the Roman numbers are not quite the same as \e{hc} and \e{harm2}, differing when $n=0$. The result is that $c_0^{(0)}=1$, differing from  $\stirc{0}{0}=0$, and for $n\lqs -1$, $k\gqs 0$ we have $c_n^{(k)} = -\stirc{n}{k} \lqs 0$.

%The  recursive definition of the usual Stirling numbers  using \e{s2} gives a very symmetrical formulation.
In summary, the array
 of numbers $\stirc{n}{k}$ appear naturally in the symmetric triple \e{xhar}, and  Theorem \ref{you} gives a very convenient way to describe them while also including both kinds of Stirling numbers.
\end{proof}

\subsection{Examples of exponential type} \label{etype}

The first example in this section is \cite[Ex. 2]{Macd}. The standard definitions of Hermite, Bernoulli  and Eulerian polynomials can be seen from the succeeding examples.%$e_k$, $h_k$ and $p_k$ are always displayed for $k\gqs 1$ with $e_0=h_0=1$.

\begin{eg}[The exponential function] \label{exp}
\begin{equation}\label{xexp}
  \left(e^{x t}, e^{x t}, x \right) \qquad \text{with} \qquad e_k=\frac{x^k}{k!}, \quad h_k=\frac{x^k}{k!}, \quad p_k= x \cdot \delta_{k,1}.
\end{equation}
\end{eg}

\begin{eg}[Hermite polynomials $\mathcal{H}_k(x)$] \label{herm}
\begin{equation}\label{xher}
  \left(e^{2x t+t^2}, \  e^{2x t-t^2}, \  2x-2t\right) \qquad \text{with} \qquad e_k=\frac{\mathcal{H}_k(ix)}{i^k k!}, \quad h_k=\frac{\mathcal{H}_k(x)}{k!},
\end{equation}
and $p_k=2x \cdot \delta_{k,1}- 2 \cdot \delta_{k,2}$. % for $\mathcal{H}_k(x)$ the $k$th Hermite polynomial.
\end{eg}

\begin{eg}[Bernoulli numbers $B_k$] \label{nou}
\begin{equation}\label{xber}
  e_k= \frac{(-1)^k}{(k+1)!}, \qquad  h_k= \frac{B_k}{k!}, \qquad  p_k= (-1)^{k-1}\frac{B_k}{k!},
\end{equation}
coming from
\begin{equation}\label{rus}
 \left( \frac{e^{-t}-1}{-t} ,  \   \frac{t}{e^t-1},  \  \frac{1}{t}\left(1-  \frac{-t}{e^{-t}-1}\right)\right).
\end{equation}
\end{eg}

\begin{proof}[Discussion]
This example corresponds to the Newton-Euler pairs in \cite[Exs. 16, 21]{Sun05} and matches \e{eqq} with $c=-1$ using \e{macd1}.
The transition formula $h_k \leftrightarrow e_k$ and \e{far} give the simple identities
\begin{equation}\label{brn}
  \frac{B_k}{k!}= \sum_{j=0}^k (-1)^j \dm_{k,j}\left(\frac{1}{2!},\frac{1}{3!},\frac{1}{4!}, \cdots \right)
  =
   (-1)^k \left|
  \begin{matrix}
  1/2!  & 1/1! & &  \\
  1/3!  & 1/2!  & 1/1! &  \\
  \vdots & & \ddots &  \\
  \scriptstyle{1/(k+1)!} & 1/k! & \dots &  1/2!
  \end{matrix}
  \right|.
\end{equation}
Many variations of \e{xber} are possible. For instance, let $t\to -t$ in \e{rus}. Then divide that triple by the same but with $t\to 2t$ to get
\begin{equation}\label{xber-sun}
  e_k= (2-2^{k+2})\frac{B_{k+1}}{(k+1)!}, \qquad  h_k= \frac{(-1)^k}{2 \cdot k!}+ \frac{\delta_{k,0}}2, \qquad  p_k= (2^k-1)\frac{B_k}{k!},
\end{equation}
corresponding to  \cite[Ex. 20]{Sun05}.
\end{proof}

Multiplying Example \ref{exp}  by Example \ref{nou}, in the sense of \e{macd4}, produces:
\begin{eg}[Bernoulli polynomials $B_k(x)$] \label{nou2}
\begin{equation}\label{xber2}
  e_k= \frac{x^{k+1}-(x-1)^{k+1}}{(k+1)!}, \qquad  h_k= \frac{B_k(x)}{k!}, \qquad  p_k= (-1)^{k-1}\frac{B_k}{k!} +x \cdot \delta_{k,1}.
\end{equation}
%for $B_k(x)$ the $k$th Bernoulli polynomial.
\end{eg}

\begin{eg}[Eulerian polynomials $A_k(x)$] \label{eup}
\begin{equation}\label{eup2}
  e_k= \frac{(1-x)^{k-1}}{k!}+\delta_{k,0} \frac{x}{x-1}, \qquad  h_k= \frac{A_k(x)}{k!}, \qquad  p_k= \frac{x A_{k-1}(x)}{(k-1)!}+ \delta_{k,1} (1-x),
\end{equation}
coming from
\begin{equation}\label{eup3}
 \left( \frac{x-e^{(1-x)t}}{x-1} ,  \   \frac{x-1}{x-e^{(x-1)t}},  \   x\frac{x-1}{x-e^{(x-1)t}} -x+1\right).
\end{equation}
\end{eg}

\begin{proof}[Discussion]
This example is based on \cite[Sect. 4]{Bre}. See \cite[Sect. 6.5]{Comtet} for details about the Eulerian polynomials (with a slightly different normalization there). They satisfy
\begin{equation} \label{frm}
   \sum_{n=1}^\infty n^k x^n =\frac{x A_k(x)}{(1-x)^{k+1}} \qquad (k \in \Z_{\gqs 0}),
\end{equation}
with $A_0(x)=A_1(x)=1$, $A_2(x)=1+x$ and $A_3(x)=1+4x+x^2$. Euler used them to compute the Riemann zeta values $\zeta(-k)$ formally, but correctly. With the $h_k \leftrightarrow e_k$ transition formula and \e{mulk}, \e{muln},
\begin{multline*} %\label{frm2}
  A_k(x)  =  k! \sum_{j=0}^k  \dm_{k,j}\left(1,\frac{x-1}{2!},\frac{(x-1)^2}{3!},\frac{(x-1)^3}{4!}, \cdots \right)\\
   =  k! \sum_{j=0}^k  (x-1)^{k-j} \dm_{k,j}\left(\frac{1}{1!},\frac{1}{2!},\frac{1}{3!},\frac{1}{4!}, \cdots \right)
   =   \sum_{j=0}^k  (x-1)^{k-j} j! \stirb{k}{j},
\end{multline*}
using \e{bakb} for the last equality. This proves Frobenius's 1910 formula for $A_k(x)$; see \cite[p. 244]{Comtet}. Example \ref{eup} and \e{eup2} are just the starting point in \cite[Chap. 3]{MR15} for an in-depth  study of permutation statistics.
\end{proof}

\begin{eg}[Cosine] \label{cs}
\begin{equation}\label{xcos}
  \left(\sec t, \cos t, -\tan t\right) \qquad \text{with} \qquad e_{2k}=\frac{U_{2k}}{(2k)!}, \quad h_{2k}=\frac{(-1)^k}{(2k)!}, \quad p_{2k}= -\frac{U_{2k-1}}{(2k-1)!}.
\end{equation}
\end{eg}

\begin{proof}[Discussion]
 See also \cite[Ex. 19]{Sun05}. The odd indexed elements are zero in this example and the next two. Here $U_n$ counts the number of alternating permutations of $n$ objects. In terms of Bernoulli polynomials and numbers we have
$$
e_{2k} = (-1)^{k-1}2^{4k+2} \frac{ B_{2k+1}(1/4)}{(2k+1)!}, \qquad p_{2k}=  (-1)^{k}\left(2^{4k}-2^{2k}\right) \frac{B_{2k}}{(2k)!}.
$$ % -\frac{U_{2k-1}}{(2k-1)!} =
The $U_{2k}$ are called secant numbers and the $U_{2k-1}$ are  tangent numbers. See \cite[Eq. (14.3)]{Ra}, \cite[p. 287]{Knu}, \cite[Thm. 3.5]{MR15} for more information.  By the $e_k \leftrightarrow h_k$ transition formula,
 \begin{equation*}%\label{sec}
  \frac{U_{2k}}{(2k)!}= \sum_{j=0}^{2k} (-1)^{k-j} \dm_{2k,j}\left(0,\frac{1}{2!},0,\frac{1}{4!},0, \cdots \right). \qedhere
\end{equation*}
\end{proof}

\begin{eg}[Sine]\label{sn}
\begin{equation}\label{xsin}
    e_{2k}=(-1)^{k-1}\frac{\left(2^{2k}-2\right) B_{2k}}{(2k)!}, \qquad h_{2k}=\frac{(-1)^k}{(2k+1)!},  \qquad p_{2k}=(-1)^{k}\frac{2^{2k} B_{2k}}{(2k)!}.
\end{equation}
coming from
$
\left(t \csc t, (\sin t)/t, -1/t+ \cot t\right).
$
\end{eg}

Example \ref{sn} relates to \cite[Exs. 18,  22]{Sun05}. Euler's formula also shows the Riemann zeta connection $p_{2k}=-2 \pi^{-2k} \zeta(2k)$. Dividing the sine triple by the cosine triple gives:

\begin{eg}[Tangent] \label{tg}
\begin{equation}\label{xtan}
    \left(t \cot t, \ \frac{\tan t}t,  \ -\frac 1t+ \tan t +\cot t\right)
\end{equation}
and
\begin{equation}\label{xtan2}
    e_{2k}=(-1)^{k}\frac{2^{2k} B_{2k}}{(2k)!}, \qquad h_{2k}=\frac{U_{2k+1}}{(2k+1)!},  \qquad p_{2k}=(-1)^{k}(2^{2k+1}-2^{4k})\frac{ B_{2k}}{(2k)!}.
\end{equation}
\end{eg}

The odd indexed coefficients are zero in Examples \ref{cs}, \ref{sn} and \ref{tg}. By Proposition \ref{root} such triples $(E,H,P)$ may be rewritten as
\begin{equation} \label{til}
 (\tilde{E}(t),\tilde{H}(t),\tilde{P}(t))= \left( E(i\sqrt{t}), \ H(\sqrt{t}), \ \frac{1}{2 \sqrt{t}}P(\sqrt{t}) \right),
\end{equation}
with
\begin{equation*}
  \tilde{e}_k = (-1)^k e_{2k}, \qquad \tilde{h}_k =  h_{2k}, \qquad \tilde{p}_k =  p_{2k}/2.
\end{equation*}

\begin{eg}[Hyperbolic sine, squared] \label{ttx}
\begin{equation}\label{xreh}
    \left(\frac t4 \sin^{-2}\left( \frac{\sqrt{t}}2 \right),   \   \frac 4t \sinh^{2}\left( \frac{\sqrt{t}}2 \right),  \
    -\frac 1t + \frac 1{2\sqrt{t}} \coth\left( \frac{\sqrt{t}}2 \right)\right)
\end{equation}
with
\begin{equation}\label{ex55}
  e_k = (-1)^k (1-2k) \frac{B_{2k}}{(2k)!}, \qquad h_k =  \frac{2}{(2k+2)!}, \qquad p_k =  \frac{B_{2k}}{(2k)!}.
\end{equation}
\end{eg}

\begin{proof}[Discussion]
This example extends the work in \cite[p. 101]{Saa}  and \cite[Ex. 17]{Sun05}. To obtain $p_k=B_{2k}/(2k)!$ we start with Example \ref{sn}, (alternatively Example \ref{nou2} with $x=1/2$ could be used). As in section \ref{calc}, let $t \to t/2$ and apply \e{til}. Then let $t \to -t$ so that $\sqrt{t} \to i\sqrt{t}$ and  the hyperbolic versions of $\sin$ and $\cot$ are needed. Squaring the resulting triple yields \e{xreh}.
The formula in \e{ex55} for $e_k$ comes from the coefficients we have already seen for $\cot(t)$ since $\cot'(t)=-\sin^{-2}(t)$. With $\sinh^2(t)=(\cosh(t)-1)/2$, the  well-known expansions of $\cosh(t)$ and $\coth(t)$ give $h_k$ and $p_k$.

The $p_k \leftrightarrow h_k$ transition formula implies the
identity, for $k\gqs 1$,
\begin{equation}\label{brn7}
  \frac{B_{2k}}{k \cdot (2k)!}= -\sum_{j=1}^k \frac{(-2)^j}j \dm_{k,j}\left(\frac{1}{4!},\frac{1}{6!},\frac{1}{8!}, \cdots \right).
\end{equation}
This is equivalent to \cite[Eq. LXXVI, p. 101]{Saa} after a correction. In the same way, from Examples \ref{cs}, \ref{sn} after applying \e{til},
\begin{align}\label{ste}
  (2^{2k}-1)2^{2k-1} \frac{B_{2k}}{k \cdot (2k)!} & = -\sum_{j=1}^k \frac{(-1)^j}j \dm_{k,j}\left(\frac{1}{2!},\frac{1}{4!},\frac{1}{6!}, \cdots \right), \\
  2^{2k-1} \frac{B_{2k}}{k \cdot (2k)!} & = -\sum_{j=1}^k \frac{(-1)^j}j \dm_{k,j}\left(\frac{1}{3!},\frac{1}{5!},\frac{1}{7!}, \cdots \right). \label{ste2}
\end{align}
These identities \e{ste} and \e{ste2} are due to Stern \cite[Eqs. (5), (9)]{Ste}.
\end{proof}

\subsection{Hypergeometric summation identities}

We would like to let $n\to \infty$ in the harmonic number Example \ref{onic}. The convergence issue at $k=1$ can be fixed by dividing by the triple $(n^t,n^t,\log n)$ first. In fact it is simpler to start with the answer, given next, where the digamma function $\psi(t)$ is defined as $\G'(t)/\G(t)$ and $p_k$ comes from the expansion \cite[Eq. (6.3.14)]{AS} for $\psi(1-t)$, valid for $|t|<1$. We also use the Euler reflection formula for the gamma function in \cite[Thm. 1.2.1]{AAR}.

\begin{eg}[Gamma function, zeta values] \label{gammax}
\begin{equation}\label{xgam}
  \left( \frac{\sin \pi t}{\pi t} \G(1-t), \  \G(1-t), \ -\psi(1-t)\right)
  \qquad \text{with} \qquad p_k= \zeta(k)
\end{equation}
for $k \gqs 2$ and $p_1 = \g$, Euler's constant.
\end{eg}

\begin{proof}[Discussion]
The transition formulas can be used to find $e_k$ and $h_k$ with, for example,
\begin{equation*}%\label{agm}
  h_k= \sum_{j=0}^k \frac{1}{j!} \dm_{k,j}\left(\g,\frac{\zeta(2)}{2},\frac{\zeta(3)}{3}, \cdots \right). \qedhere
\end{equation*}
\end{proof}

Examples \ref{onic} and \ref{gammax} give the relations needed to develop the hypergeometric summation identities of Chu \cite{Chu97}. The idea is to expand hypergeometric summation formulas into power series and compare coefficients on both sides. We may illustrate this with Gauss's summation formula, \cite[Thm. 2.2.2]{AAR}:
\begin{equation}\label{gau}
  \sum_{n=0}^\infty \frac{(x)_n (y)_n}{n! (1-z)_n} =: {_2}F_1(x,y;1-z;1) = \frac{\G(1-z) \G(1-x-y-z)}{\G(1-x-z) \G(1-y-z)},
\end{equation}
where $(x)_n$ means $x(x+1) \cdots (x+n-1)$. % and $\Re(z-x-y)>0$.
With $p_k$ as in Example  \ref{gammax}, the right side of \e{gau} is
\begin{equation} \label{pxe}
  \exp\left( \sum_{k=1}^\infty \frac{p_k}{k} \left[ z^k +(x+y+z)^k-(x+z)^k-(y+z)^k\right]\right).
\end{equation}
The left side is
\begin{equation} \label{lef}
  1+x y \sum_{n=0}^\infty  \frac{\left(1+\frac{x}{1} \right) \cdots \left(1+\frac{x}{n} \right) \cdot \left(1+\frac{y}{1} \right) \cdots \left(1+\frac{y}{n} \right)}{(n+1)^2 \left(1-\frac{z}{1} \right) \cdots \left(1-\frac{z}{n+1} \right)}
\end{equation}
and we recognize $E$ and $H$ from Example \ref{onic}.

To make the example simpler, let $y=x$ and $z=0$. The $e_k \leftrightarrow p_k$ transition formula for the square of the triple in Example \ref{onic} gives
\begin{equation*}
  [x^m] \left(\left(1+\frac{x}{1} \right) \cdots \left(1+\frac{x}{n} \right)\right)^2 = \sum_{\ell=0}^m (-1)^{m-\ell} \frac{1}{\ell!} \dm_{m,\ell}\left(\frac{2H_n^{(1)}}{1},\frac{2H_n^{(2)}}{2}, \frac{2H_n^{(3)}}{3}, \cdots \right).
\end{equation*}
Comparing this with the expansion of \e{pxe} by \e{cox} proves the following.

\begin{prop} \label{nov}
For $m\gqs 0$,
\begin{multline*}
  \sum_{\ell=0}^m (-1)^{m-\ell} \frac{2^\ell}{\ell!} \sum_{n=0}^\infty \frac{1}{(n+1)^2} \dm_{m,\ell}\left(\frac{H_n^{(1)}}{1},\frac{H_n^{(2)}}{2}, \frac{H_n^{(3)}}{3}, \cdots \right) \\
  = \sum_{j=0}^{m+2} \frac{1}{j!} \dm_{m+2,j}\left(0,(2^2-2)\frac{\zeta(2)}{2},(2^3-2)\frac{\zeta(3)}{3}, (2^4-2)\frac{\zeta(4)}{4}, \cdots \right).
\end{multline*}
\end{prop}

This generates a harmonic number identity for each $m$. When $m=1,2,3$ we find
\begin{align}\label{eu}
   \sum_{n=0}^\infty \frac{H_n}{(n+1)^2} & = \zeta(3),\\
   \sum_{n=0}^\infty \frac{2H^2_n- H_n^{(2)}}{(n+1)^2} & = \frac 12\zeta(2) + \frac 72\zeta(4) = \frac{19 \pi^4}{360},\label{eu2}\\
  \sum_{n=0}^\infty \frac{2H^3_n- 3H_n H_n^{(2)}+H_n^{(3)}}{(n+1)^2} & = 3\zeta(2)\zeta(3) + \zeta(5). \label{eu3}
\end{align}
The identity \e{eu} is \cite[Eq. (1.1a)]{Chu97}, and  due to Euler, with the left side being the multiple zeta value $\zeta(2,1)$. See \cite{Chu97}, and the many subsequent papers on this topic, for examples of the wide variety of identities that can be produced in this way. Similar identities appear in \cite{Bat}. The novelty of  Proposition \ref{nov} is that it shows the structure of all the identities produced at once, at least in this one-variable example. The paper \cite{Hoff} has  more on the connections between symmetric functions and multiple zeta values.

The already very elegant identity \cite[Eq. (5.5)]{Chu97} has an even more succinct statement with our notation from Theorem \ref{you}: %, (recall that $\stirc{-n}{q} = \stira{n}{q}/(n-1)!$):

\begin{prop}
For all  integers  $p\gqs 0$ and $q\gqs 1$,
\begin{equation}\label{pqu}
  \sum_{n=1}^\infty \frac{1}{n^2} \stirc{n}{p} \cdot \stirc{-n}{q} = \binom{p+q}{p} \zeta(p+q+1).
\end{equation}
\end{prop}

A version of the proof from \cite{Chu97} is  included next, since we have already done the groundwork. This result is also Theorem 4 of \cite{Hoff}, where it has a different proof.

\begin{proof}  Subtracting $1$ from \e{lef}, dividing by $x$ and letting $x\to 0$ makes
\begin{equation} \label{lef2}
   y \sum_{n=0}^\infty  \frac{ \left(1+\frac{y}{1} \right) \cdots \left(1+\frac{y}{n} \right)}{(n+1)^2 \left(1-\frac{z}{1} \right) \cdots \left(1-\frac{z}{n+1} \right)}
   =\sum_{u=0}^n \sum_{v=0}^\infty y^{u+1} z^v \sum_{n=0}^\infty  \frac{e_u(n) h_v(n+1)}{(n+1)^2}
\end{equation}
with $e_k(n)$ and $h_k(n)$ from \e{xhar} or equivalently \e{xxhar}. Doing the same on the right of \e{gau} yields
\begin{align}
  -\frac{\G'(1-y-z)}{\G(1-y-z)}+ \frac{\G'(1-z)}{\G(1-z)} & =\psi(1-y-z)+\psi(1-z) \notag \\
  & = \sum_{k=1}^\infty p_k \left( (y+z)^{k-1} - z^{k-1}\right), \label{rgt}
\end{align}
with $p_k$ from Example  \ref{gammax}. Comparing powers of $y$ and $z$ in \e{lef2} and \e{rgt} finishes the argument.
\end{proof}

\section{$q$-Series}
\subsection{$q$-Binomial coefficients}
Define
\begin{equation*}
  (a;q)_n:= (1-a)(1-aq)(1-aq^2)\cdots (1-aq^{n-1}).
\end{equation*}
The next result is fundamental and due to Cauchy.
\begin{theorem}[The $q$-binomial theorem] \label{qbin}
For $|q|$, $|t|<1$ and all $a$,
\begin{equation}\label{qm}
  \sum_{m=0}^\infty \frac{(a;q)_m}{(q;q)_m} t^m = \frac{(a t;q)_\infty}{(t;q)_\infty}.
\end{equation}
\end{theorem}
See \cite[Thm. 2.1]{And76} or \cite[Thm. 10.2.1]{AAR} for the short proof. The $q$-binomial coefficient is defined to be
\begin{equation*}
 \binom{n}{k}_{\!q} := \frac{(q;q)_n}{(q;q)_k (q;q)_{n-k}}.
\end{equation*}
Then, as in \cite[Cor. 10.2.2]{AAR}, the special cases $a=q^{-n}$ and $a=q^{n}$ of Theorem \ref{qbin} imply
\begin{equation} \label{tqn}
  (t;q)_n = \sum_{k=0}^n (-1)^k q^{\binom{k}{2}} \binom{n}{k}_{\!q} t^k, \qquad \frac 1{(t;q)_n} = \sum_{k=0}^n \binom{n+k-1}{k}_{\!q} t^k.
\end{equation}
Therefore, choosing $x_j=q^{j-1}$ in \e{genf} and \e{genf2} yields the next triple.

\begin{eg}[$q$-Binomial coefficients] \label{qbc} For $n \gqs 1$,
\begin{equation}\label{qb}
  e_k = q^{\binom{k}{2}}\binom{n}{k}_{\!q}, \qquad h_k = \binom{n+k-1}{k}_{\!q},  \qquad p_k  = \frac{1-q^{nk}}{1-q^k}.
\end{equation}
\end{eg}

\begin{proof}[Discussion] The convolution and transition formulas give relations between $e_k$, $h_k$ and $p_k$ as rational functions of $q$. For example, when $n=2$ we find the polynomial identity
\begin{equation}\label{mow}
  \sum_{k=0}^m \frac{1}{k!} \dm_{m,k}\left(\frac{1+q}1, \frac{1+q^2}2, \frac{1+q^3}3, \dots \right)  = 1+q+q^2+ \cdots + q^m,
\end{equation}
from the $h_k \leftrightarrow p_k$ transition formula.
Taking the limit as $q\to 1^-$ in Example \ref{qbc} gives \e{nck}, a special case of Example \ref{ex-bc}.
\end{proof}

We may also take the limit as $n\to \infty$ of Example \ref{qbc} when $|q|<1$. %The resulting symmetric triple is valid for all $q$...:

\begin{eg}[Limit of $q$-binomial coefficients] \label{qbc2}
\begin{equation}\label{qb2}
  e_k = q^{\binom{k}{2}}\frac{1}{(q;q)_k}, \qquad h_k = \frac{1}{(q;q)_k},  \qquad p_k  = \frac{1}{1-q^k}.
\end{equation}
\end{eg}

Examples \ref{qbc} and \ref{qbc2} are given in \cite[Ex. 3, Ex. 4]{Macd}. They are special cases of the next symmetric triple, which is \cite[Ex. 5]{Macd}.

\begin{eg}[$q$-Series with two-parameters]
\begin{equation}\label{qbs}
  e_k = \frac{(-1)^k}{(q;q)_k} \prod_{j=0}^{k-1} (b-a q^j), \qquad h_k = \frac{1}{(q;q)_k} \prod_{j=0}^{k-1} (a-b q^j),  \qquad p_k  = \frac{a^k-b^k}{1-q^k}.
\end{equation}
\end{eg}

%\begin{proof}[Discussion] \end{proof}

\begin{proof}[Discussion]
This follows from setting
\begin{equation}\label{vg}
  H(t):= \frac{(b t;q)_\infty}{(a t;q)_\infty} = \sum_{m=0}^\infty a^m \frac{(b/a;q)_m}{(q;q)_m} t^m
\end{equation}
where the equality in \e{vg} is obtained by letting $t\to at$, $a\to b/a$ in Theorem \ref{qbin}.
A calculation finds
\begin{align}
  \log (t;q)_\infty  = \sum_{j=0}^\infty \log (1-t q^j)
   &  = -\sum_{j=0}^\infty \sum_{m=1}^\infty \frac{(t q^j)^m}{m} \notag\\
   & = - \sum_{m=1}^\infty \frac{t^m}{m} \sum_{j=0}^\infty q^{jm}
   = - \sum_{m=1}^\infty \frac{t^m}{m(1-q^m)}. \label{mov}
\end{align}
Hence
\begin{equation*}
  P(t)=\frac{d}{dt} \log H(t)= \sum_{m=1}^\infty \frac{a^m-b^m}{1-q^m} t^{m-1},
\end{equation*}
and this verifies the formula for $p_k$ in \e{qbs}.
\end{proof}

For future use, the Jacobi triple product identity is
\begin{equation}\label{jac}
  (t;q)_\infty (q/t;q)_\infty (q;q)_\infty = \sum_{k \in \Z} (-1)^k q^{\binom{k}{2}}  t^k,
\end{equation}
and can be made to follow from the left identity in \e{tqn}; see \cite[Thm. 10.4.1]{AAR}.

\subsection{Partitions}
For $S \subseteq \Z_{\gqs 1}$ and any $r$, the generating function
\begin{equation}\label{genpart}
  P_{S,r}(q):= \prod_{j\in S} \frac 1{(1-q^j)^r} = \sum_{n=0}^\infty p_{S,r}(n) q^n,
\end{equation}
defines the numbers $p_{S,r}(n)$. In the common case $r=1$,  $p_{S,1}(n)$ counts the number of partitions of $n$ with parts in $S$. When $r \in \Z_{\gqs 1}$, $p_{S,r}(n)$ is the number of $r$-color partitions of $n$ with parts in $S$. These are partitions of $n$ where each part is colored and the order of the colored parts  does not matter. Equivalently, $p_{S,r}(n)$ counts the number of $r$-component multipartitions of $n$ with parts in $S$. These are $r$-tuples of usual partitions that sum to $n$ in total.

We have
\begin{multline} \label{um}
  \log P_{S,r}(q)  =-r \sum_{j \in S} \log(1-q^j)
   = r \sum_{j \in S} \sum_{n=1}^\infty \frac{q^{nj}}{n} \\
   = r  \sum_{m=1}^\infty q^m \sum_{j \in S, \ j | m} \frac{j}{m} = r  \sum_{m=1}^\infty \sigma_S(m) \frac{q^m}{m}
   \qquad \text{for} \qquad \sigma_S(m):= \sum_{j \in S, \ j | m} j.
\end{multline}
Therefore $$\frac{d}{dq} \log P_{S,r}(q) = r  \sum_{m=1}^\infty \sigma_S(m) q^{m-1}$$
and we obtain from $P_{S,r}(q)$ the following symmetric triple.

\begin{eg}[Partitions, divisors] \label{parx}
\begin{equation}\label{xpadi}
  e_n= (-1)^n p_{S,-r}(n), \qquad  h_n= p_{S,r}(n), \qquad  p_n= r \cdot \sigma_S(n).
\end{equation}
\end{eg}

\begin{proof}[Discussion]
Some special cases of this may be examined more closely; see also \cite[Ex. 6]{Macd}, \cite[Exs. 5 - 9]{Sun05}.
These cases all have $S=\Z_{\gqs 1}$ and $S$ is omitted from the notation.
When  $r=1$ we find the symmetric triple
\begin{equation}\label{eule}
  e_n= (-1)^n c(n), \qquad  h_n= p(n), \qquad  p_n=  \sigma(n)
\end{equation}
from the introduction.
Recall Jacobi's identity, \cite[Eq. (10.4.9)]{AAR}:
\begin{equation} \label{ds}
  \prod_{j=1}^\infty (1-q^j)^3 =  \sum_{m=0}^\infty d(m) q^m, \qquad \text{with} \qquad
  d(m) := \begin{cases}
  (-1)^r(2r+1) & \text{ if $m=\frac{r(r+1)}{2}$ for  $r\in \Z_{\gqs 0}$},\\
  0 & \text{ otherwise. }
  \end{cases}
\end{equation}
Hence, for  $r=3$ we obtain the symmetric triple
\begin{equation}\label{eule2}
  e_n= (-1)^n d(n), \qquad  h_n= p_{3}(n), \qquad  p_n=  3\sigma(n).
\end{equation}
For  $r=24$ and Ramanujan's tau function $\tau(n)$, see \cite[Eq. (5.3)]{odm}, we find
\begin{equation}\label{eule3}
  e_n= (-1)^n \tau(n+1), \qquad  h_n= p_{24}(n), \qquad  p_n=  24\sigma(n).
\end{equation}
Some of the many identities for these last triples appear in section 5 of \cite{odm}.
\end{proof}

The transition formulas expressing $e_n$ in terms of $p_n$  above are a special cases of Bell's following result.

\begin{theorem}[Partition polynomials] \cite{Bell27} \label{bb27}
Fix   a positive integer $s$ and sets of positive integers $C_1$, $C_2, \dots, C_s$. Also fix complex numbers $a_1, \dots, a_s$ and $z_1, \dots, z_s$. Then we have the formal power series expansion
\begin{equation} \label{bepr}
  \prod_{j=1}^s \prod_{n_j \in C_j} (1-z_j \cdot q^{n_j})^{a_j} = \sum_{n=0}^\infty \Psi(n) q^n
\end{equation}
where
\begin{equation} \label{bepr2}
  \Psi(n) = \sum_{k=0}^n \frac{1}{k!} \dm_{n,k}\left(\frac{\psi(1)}{1}, \frac{\psi(2)}{2}, \dots  \right), \qquad
  \psi(n) = -\sum_{j=1}^s a_j \sum_{d \mid n, \ d\in C_j} d \cdot z_j^{n/d}.
\end{equation}
\end{theorem}
\begin{proof}
Take the logarithm of the product in \e{bepr} and expand as in \e{um}. Exponentiating with \e{cox} then completes the proof.
\end{proof}

Bell termed $\Psi(n)$, given in \e{bepr2}, the {\em partition polynomial of rank $n$ in the variables $z_1, \dots, z_s$ and associated to $C_1, \dots, C_s$ and $a_1, \dots, a_s$}. From \e{bell2} we see that $\Psi(n)$ is of degree at most $n$ when considered as a polynomial in each of $z_1, \dots, z_s$ or $a_1, \dots, a_s$ or $\psi(1), \dots, \psi(n)$.
Theorem \ref{bb27} shows that, for any identity of the form \e{bepr}, the coefficients $\Psi(n)$ can  be expressed as a polynomial of divisor sums.

As an example, take Ramanujan's famous identity \cite[Eq. (105.1)]{Ra}
\begin{equation}\label{ramj}
 5 \frac{(q^5;q^5)_\infty^5}{(q;q)_\infty^6} = \sum_{n=0}^\infty p(5n+4) q^n,
\end{equation}
showing that $5$ divides $p(5n+4)$. Applying the theorem with $C_1=\{5,10, \dots\}$, $C_2=\{1,2, \dots\}$, $a_1=5$, $a_2=-6$ and $z_1=z_2=1$ means
\begin{equation} \label{whe}
  \psi(n) = -5 \sum_{d | n,\ 5 | d} d +6 \sum_{d | n} d = \sigma(n) +5 \omega_5(n)
\end{equation}
where, in general,  $\omega_r(n)$ is defined to be the sum of all the divisors of $n$ that are not multiples of $r$. Then
\begin{equation}\label{pbe}
  p(5n+4) = 5 \sum_{k=0}^n \frac{1}{k!} \dm_{n,k}\left(\frac{\sigma(1) +5 \omega_5(1)}{1}, \frac{\sigma(2) +5 \omega_5(2)}{2}, \dots  \right).
\end{equation}
An equivalent formula to \e{pbe} is the main result of \cite{bo}.

For the corresponding symmetric triple, factor the reciprocal of the product in \e{ramj} into
\begin{equation} \label{ffx}
  \frac{(q;q)_\infty^6}{(q^5;q^5)_\infty^5} = (q;q)_\infty^3 \cdot (q;q)_\infty^3 \cdot \frac{1}{(q^5;q^5)_\infty^5}.
\end{equation}
The coefficients  do not seem to have been studied, but we may express them with \e{ffx}, \e{genpart} and \e{ds} as
\begin{equation} \label{ffx2}
  f(n):= [q^n] \frac{(q;q)_\infty^6}{(q^5;q^5)_\infty^5} =\sum_{i+j+5k=n} d(i) \cdot d(j) \cdot p_5(k).
\end{equation}

\begin{eg}[Partitions modulo $5$] %\label{parx}
\begin{equation*} %\label{xpadi}
  e_n= (-1)^n f(n), \qquad  h_n= p(5n+4)/5, \qquad  p_n= \sigma(n) +5 \omega_5(n).
\end{equation*}
\end{eg}

\subsection{Sums of squares or triangular numbers} \label{sq-tri}

With the theta function
 definition
\begin{equation}\label{the}
  \varphi(q):=\sum_{n \in \Z} q^{n^2}, \qquad \text{then} \qquad \varphi(-q)=\sum_{n \in \Z} (-1)^n q^{n^2}
  = \prod_{j=1}^\infty \frac{1-q^j}{1+q^j}
\end{equation}
by \e{jac} when $t\to q$, $q\to q^2$. We may set $E(q)=\varphi(q)$ and $H(q)=1/\varphi(-q)$. Directly as in \e{um}, or by Theorem \ref{bb27},
\begin{equation*}
  \log H(q) =  \sum_{m=1}^\infty (\sigma(m) +\omega_2(m))\frac{q^{m}}{m}
\end{equation*}
where $\omega_2(n)$ is the sum of the odd divisors of $n$. The series $H(q)$ is the generating function for the number $\overline p(n)$ of overpartitions of $n$. These are partitions where the first appearance of a part of each size may or may not be overlined.

\begin{eg}[Squares, overpartitions] \label{over}
For $n \gqs 1$,
\begin{equation}\label{exop}
  e_n = \begin{cases}
  2 & \text{ if }n=m^2;\\
  0 & \text{ if }n\neq m^2,
  \end{cases}
  \qquad h_n = \overline p(n), \qquad p_n =  \sigma(n) +\omega_2(n).
\end{equation}
\end{eg}

We may generalize Example \ref{over} by taking the $r$th power. Write
\begin{equation}\label{rpw}
  \varphi(q)^r =\sum_{n =0}^\infty N_r(n) q^{n}, \qquad \prod_{j=1}^\infty \left(\frac{1+q^j}{1-q^j}\right)^r =\sum_{n =0}^\infty \overline{p}_r(n) q^{n}
\end{equation}
for any $r$. When $r$ is a positive integer, $N_r(n)$ is the number of ways to express $n$ as a sum of squares of $r$ integers (including the order of the integers and their signs), while $\overline{p}_r(n)$ naturally counts $r$-colored overpartitions. These are $r$-colored partitions where the first appearance of a part of each size and color may be overlined.

\begin{eg}[Representations as sums of squares] \label{sqrs}
\begin{equation}\label{exrep}
  e_n = N_r(n),
  \qquad h_n = \overline{p}_r(n), \qquad p_n =  r(\sigma(n) +\omega_2(n)).
\end{equation}
\end{eg}
\begin{proof}[Discussion]
This is Example \ref{bell-eu2} again. See also \cite[Ex. 22]{Macd}, \cite[Ex. 10]{Sun05}.
By the transition formulas, we obtain \e{bex4} and
\begin{equation}
  N_r(n)  = \sum_{k=0}^{n} (-1)^{n-k}  \dm_{n,k}\left(\overline{p}_r(1),\overline{p}_r(2),\overline{p}_r(3), \dots \right). \label{repsq2}
\end{equation}
A nice simplification of \e{repsq2} is possible, which may also be inverted:

\begin{prop} \label{pj}
For $n \in \Z_{\gqs 0}$ and arbitrary $r$,
\begin{equation}\label{fs}
  N_r(n) = \sum_{j=0}^{n} (-1)^{n-j} \binom{n+1}{j+1} \overline{p}_{r j}(n),\qquad
  \overline{p}_{r}(n)  = \sum_{j=0}^{n} (-1)^{n-j} \binom{n+1}{j+1} N_{r j}(n).
\end{equation}
\end{prop}
\begin{proof}
We have
\begin{equation*}
  \dm_{n,k}\left(1,\overline{p}_r(1),\overline{p}_r(2),\overline{p}_r(3), \cdots \right) = [q^n] q^k
  \prod_{j=1}^\infty \left(\frac{1+q^j}{1-q^j}\right)^{r k}
  = \overline{p}_{rk}(n-k).
\end{equation*}
Hence, by \e{add},
\begin{align*}
  \dm_{n,k}(\overline{p}_r(1),\overline{p}_r(2), \dots) & = \sum_{j=0}^k (-1)^{k-j} \binom{k}{j} \dm_{n+j,j}(1,\overline{p}_r(1),\overline{p}_r(2), \dots)\\
  & = \sum_{j=0}^k (-1)^{k-j} \binom{k}{j} \overline{p}_{rk}(n-k).
\end{align*}
Using this in \e{repsq2} and simplifying with an elementary binomial identity (\cite[p. 171]{Knu}) gives the
left formula in \e{fs}. The right one is proved the same way.
\end{proof}

The $r=1$ case of the right formula of Proposition \ref{pj} has recently appeared in \cite{Jha21} giving a simple link between overpartitions and representations as sums of squares:
\begin{equation}\label{ssq}
  \overline{p}(n)  = \sum_{j=0}^{n} (-1)^{n-j} \binom{n+1}{j+1} N_{j}(n).
\end{equation}
There are formulas for $\sigma(n) +\omega_2(n)$ of the same kind as \e{fs}, (see \e{jjx}). An example of a convolution identity here is
\begin{equation} \label{cvi}
  r(\sigma(n) +\omega_2(n)) = \sum_{j=1}^n (-1)^{j-1} j \cdot N_r(j)  \cdot \overline{p}_{r}(n-j), %\qedhere
\end{equation}
from \e{peh}.
\end{proof}

Next, take $\psi(q):= \sum_{n=0}^\infty q^{n(n+1)/2}$ with
\begin{equation*}
    \frac 1{\psi(-q)}=
   \prod_{j=1}^\infty \frac{1+q^{2j-1}}{1-q^{2j}} =
   \prod_{j=1}^\infty \frac{1}{(1-q^{4j})(1-q^{2j-1}}
    = \prod_{j\gqs 1, \ j \not\equiv 2 \bmod 4} \frac{1}{1-q^{j}}.
\end{equation*}
Let $\hat p(n)$ be the number of partitions of $n$ where no parts have size $\equiv 2 \bmod 4$, and recall $\omega_r(n)$ defined after \e{whe}. Then similarly to Example \ref{over} we find the symmetric triple:

\begin{eg}[Triangular numbers] \label{trix}
%For $n \gqs 1$,
\begin{equation}\label{extrx}
  e_n = \begin{cases}
  1 & \text{ if $n=\frac{m(m+1)}{2}$};\\
  0 & \text{ otherwise},
  \end{cases}
  \qquad h_n = \hat p(n), \qquad p_n =  \sigma(n) + \omega_2(n) - \omega_4(n).
\end{equation}
\end{eg}

Raising this to the power $r$, write
\begin{equation}\label{trp}
  \psi(q)^r =\sum_{n =0}^\infty \Delta_r(n) q^{n}, \qquad \prod_{j\gqs 1, \ j \not\equiv 2 \bmod 4} \frac{1}{(1-q^{j})^r} =\sum_{n =0}^\infty \hat{p}_r(n) q^{n}.
\end{equation}
When $r$ is a positive integer, $\Delta_r(n)$ is the number of ways to express $n$ as a sum of triangular numbers (including the order of the numbers), while $\hat{p}_r(n)$ gives an $r$-colored version of $\hat{p}(n)$.

\begin{eg}[Representations as sums of triangular numbers] \label{angx}
\begin{equation}\label{exrep6}
  e_n = \Delta_r(n),
  \qquad h_n = \hat{p}_r(n), \qquad p_n =  r (\sigma(n) + \omega_2(n) - \omega_4(n)).
\end{equation}
\end{eg}

\begin{proof}[Discussion]
Analogs of all the identities we saw for Example \ref{sqrs} are available here too. For example
\begin{equation}\label{jjx}
  \frac{r}n(\sigma(n) + \omega_2(n) - \omega_4(n)) =\sum_{j=1}^n \frac{(-1)^{n-j}}j \binom{n}{j} \Delta_{r j}(n),
\end{equation}
and this one, for $r=1$, is equivalent to the main result in \cite{Jhatri}. Several of Jha's recent papers can be organized and understood by referring to Examples \ref{parx}, \ref{sqrs} and \ref{angx}.
It is also interesting to consider generalizations of the last  examples, where the  squares and triangular numbers are replaced by other sets of integers $S$.
\end{proof}

\subsection{Further identities}

We may give a De Moivre polynomial version of the $q$-binomial theorem.
Utilizing \e{mov} shows that
\begin{equation*}
  \log \frac{(a z;q)_\infty}{(z;q)_\infty} = \sum_{m=1}^\infty \frac{1-a^m}{m(1-q^m)} z^m,
\end{equation*}
and so Theorem \ref{qbin} is equivalent to
\begin{equation}\label{qid}
  \sum_{k=0}^m \frac{1}{k!} \dm_{m,k}\left(\frac{1-a}{1(1-q)}, \frac{1-a^2}{2(1-q^2)}, \frac{1-a^3}{3(1-q^3)}, \dots \right) = \frac{(a;q)_m}{(q;q)_m}.
\end{equation}
Special cases of this are, for $m\gqs 0$,
\begin{align}\label{cau}
  \sum_{k=0}^m \frac{1}{k!} \dm_{m,k}\left(\frac 1{1}, \frac 1{2}, \frac 1{3}, \dots \right) & = 1, \\
  \sum_{k=0}^m \frac{1}{k!} \dm_{m,k}\left(\frac{1-a}1, \frac{1-a^2}2, \frac{1-a^3}3, \dots \right) & = 1-a \qquad (m\gqs 1), \label{cau2}\\
  \sum_{k=0}^m \frac{1}{k!} \dm_{m,k}\left(\frac{1}{1(1-q)}, \frac{1}{2(1-q^2)}, \frac{1}{3(1-q^3)}, \dots \right) & = \frac{1}{(q;q)_m}, \label{cau3}
\end{align}
where \e{cau} is a result of Cauchy and \e{cau2} is \cite[Eq. (11)]{Mo71}.
The right side of \e{cau3} is the generating function for partitions into at most $m$ parts and MacMahon used \e{cau3} to simplify expressions for them \cite[Vol. 2, p. 62]{Macm}.
Similar identities to \e{cau} arise from the following symmetric triple, as in \cite[Ex. 13]{Macd}.

\begin{eg}[Schur's identity and generalizations] For $r$ a positive integer,
$$
\left( \frac{(1+t)^r}{1-(-t)^r},   \  \frac{1-t^r}{(1-t)^r},   \  \frac{r}{1-t}-\frac{r t^{r-1}}{1-t^r}\right) \qquad \text{with} \qquad p_n =  \begin{cases}
  0 & \text{ if }r \mid n;\\
  r & \text{ if }r\nmid n.
  \end{cases}
$$
\end{eg}
%\begin{proof}[Discussion]  \end{proof}

\begin{proof}[Discussion]
The binomial theorem  implies for $n\gqs 1$
\begin{equation}\label{mor}
  e_n =  (-1)^{r\lfloor n/r \rfloor} \binom{r}{n-r \lfloor n/r \rfloor} + \begin{cases}
  (-1)^{n-r} & \text{ if }r \mid n;\\
  0 & \text{ if }r\nmid n,
  \end{cases}
  \qquad h_n = \binom{n+r-1}{r-1}-\binom{n-1}{r-1}.
\end{equation}
By the transition formula $h_k \leftrightarrow p_k$,
\begin{equation}\label{mor2}
  \sum_{k=0}^n \frac{r^k}{k!} \dm_{n,k}\left(\frac 1{1}, \frac 1{2}, \dots, \frac 1{r-1},0, \frac 1{r+1}, \dots\right)  = \binom{n+r-1}{r-1}-\binom{n-1}{r-1},
\end{equation}
where the entries in $ \dm_{n,k}$ at a multiple of $r$ are set to zero. Formula \e{mor2}  is due to Morris, \cite[Eq. (13)]{Mo71}, generalizing the $r=2$ case which is an identity of Schur:
\begin{equation*}%\label{mor3}
  \sum_{k=0}^n \frac{2^k}{k!} \dm_{n,k}\left(\frac 1{1}, 0, \frac 1{3}, 0, \frac 1{5},0, \dots \right)  = 2, \qquad (n\gqs 1).
  \qedhere
\end{equation*}
\end{proof}

For another interesting case with rational functions, take $E(t)$ to be a quadratic polynomial.
\begin{eg}[Lucas sequences] \label{luc}
$$
\left( 1+\alpha t +\beta t^2,   \  \frac{1}{1-\alpha t +\beta t^2},   \  \frac{\alpha -2 \beta t}{1-\alpha t +\beta t^2}\right).
$$
\end{eg}
\begin{proof}[Discussion]
Then $h_n$ and $p_n$ give the usual basis for Lucas sequences: any sequence $\ell_n$ satisfying the recursion $\ell_{n+1} = \alpha \ell_n -\beta \ell_{n-1}$ must equal $A \cdot h_n+B  \cdot p_{n+1}$ for some fixed $A$ and $B$. Initial conditions for $h_n$ and $p_n$ can be taken to be $h_{-1}=0$, $h_0=1$, $p_0=2$ and $p_1=\alpha$.
The general theory is described in \cite[Sect. 7.5]{chacha}.

By the $h_n \leftrightarrow e_n$ transition formula, %as in \cite[Eq. (5.13)]{odm},
\begin{equation*}
  h_n = \sum_{k=0}^n (-1)^{n-k} \dm_{n,k}(\alpha,\beta,0,0, \dots) = \sum_{k=0}^n \binom{k}{n-k} \alpha^{2k-n} (-\beta)^{n-k},
\end{equation*}
with a similar formula for $p_n$. See \cite[Ex. 15]{Sun05} for more on the divisibility of such sequences.
%The triple in \cite[Ex. 17]{Macd} also has rational functions.
\end{proof}

%Perhaps surprisingly,
Chebyshev polynomials of the first and second kind, $T_n(x)$ and $U_n(x)$ respectively, appear in the special case $\alpha=2x$, $\beta=1$ of Example \ref{luc}. See also \cite[Sect. 5]{odm} and \cite[Exer. 3.16]{MR15}.

\begin{eg}[Chebyshev polynomials $T_n(x)$, $U_n(x)$] \label{cheb}
$$
E(t)=1+2x t+t^2,  \qquad h_n=U_n(x),  \qquad p_n=2T_n(x).
$$
\end{eg}

\section{Compositional inverses} \label{cinv}

Let $f(t)$ be a formal power series with a  constant term $c$ that is invertible in our ring of coefficients $R$. Though $f(t)$ may not have a compositional inverse,  $F(t):=t f(t)$ must have one: $G(x)=F(x)^{\langle -1\rangle}$ so that  $G(x)=t \iff F(t)=x$. See \cite[Prop. 3.8]{odm}, for example. Then $f^*(t):=G(t)/t$ is a power series with constant term $1/c$. Following \cite[Ex. 24]{Macd}, this defines the $*$ operation on power series with invertible constant terms:
\begin{equation}\label{star}
  f^*(t):=(t f(t))^{\langle -1\rangle}/t.
\end{equation}

 If $(E(t),H(t),P(t))$ is a symmetric triple, then $H^*(t)$ is a series with constant term $1$ and we obtain the 'starred' symmetric triple
\begin{equation}\label{invh}
  (E,H,P)^*:=\left(E^\circ(t),H^*(t),P^\circ(t)\right) \qquad \text{for} \qquad E^\circ(t):=\frac{1}{H^*(-t)}, \quad P^\circ(t):= \frac d{dt}\log H^*(t).
\end{equation}
Writing the coefficients of the series in \e{invh} as $e_n^\circ$, $h_n^*$ and $p_n^\circ$, we
find, as the special case $\alpha=\beta=-1$ of Corollary \ref{newtrip}, the remarkably similar formulas for $n\gqs 1$:
\begin{equation}\label{rem}
  (-1)^{n-1} e_n^\circ = [t^n] \frac{H(t)^{-(n-1)}}{n-1},
  \qquad
  h_n^* = [t^n] \frac{H(t)^{-(n+1)}}{n+1},
  \qquad
  \frac{p_n^\circ}{n} = [t^n] \frac{H(t)^{-n}}{n}.
\end{equation}
Applying the $*$ operation again also shows that
\begin{equation}\label{rem2}
  (-1)^{n-1} e_n = [t^n] \frac{H^*(t)^{-(n-1)}}{n-1},
  \qquad
  h_n = [t^n] \frac{H^*(t)^{-(n+1)}}{n+1},
  \qquad
  \frac{p_n}{n} = [t^n] \frac{H^*(t)^{-n}}{n}.
\end{equation}
The formulas for $e_1^\circ$ and $e_1$ above should be interpreted using \e{zr}. More simply, with \e{ones} they can be replaced by $h_1^*$ and $h_1$ respectively.

Then \e{rem} can be used to relate $e_n^\circ$, $h_n^*$ and $p_n^\circ$ back to the original $e_n$, $h_n$ and $p_n$.
To do this, express $H$ in terms of $E$ or $P$ as in the proof of Theorem \ref{tran} and then expand using Proposition \ref{comp}. This yields the following, as a De Moivre polynomial version of \cite[Ex. 24]{Macd}.

\begin{prop} \label{invsy}
For $n \gqs 1$,
\begin{gather*}\label{ese}
e_n^\circ =  \frac{-1}{n-1} \sum_{k=0}^n \binom{n-1}{k} \dm_{n,k}(e_1,e_2,\dots),
\qquad
  h_n^* = \frac{1}{n+1} \sum_{k=0}^n \binom{-n-1}{k} \dm_{n,k}(h_1,h_2, \dots), \\
  p_n^\circ =  \sum_{k=0}^n \frac{(-n)^k}{k!} \dm_{n,k}\left(\frac{p_1}{1},\frac{p_2}{2},\dots \right). \label{psp}
\end{gather*}
\end{prop}

%\begin{cor}
%If $H^*(t)-1=c( H(t)-1)$ for any constant $c$, then $H(t)= 1$.
%\end{cor}
%\begin{proof}

%\end{proof}

\begin{eg}[The logarithm] \label{logx}
$$
\left( \frac{t}{\log(1+t)},  \ \frac{\log(1-t)}{-t},  \ -\frac{1}{t}-\frac{1}{(1-t)\log(1-t)}\right),
$$
with
\begin{equation} \label{sgns}
  e_n=\int_0^1 \binom{x}{n} \, dx, \qquad h_n=\frac{1}{n+1}, \qquad p_n=(-1)^n \int_0^1 \binom{-x}{n}\, dx.
\end{equation}
\end{eg} % \frac{\mathcal{C}^-_n}{n!} :   \frac{\mathcal{C}^+_n}{n!} :=

\begin{proof}[Discussion]
This example comes from applying \e{macd1} to the Bernoulli Example \ref{nou}, so that $H(t)$ becomes $(e^{-t}-1)/(-t)$, and then applying the $*$ operation to find $H^*(t)=-\log(1-t)/t$ directly. The formulas for $e_n$ and $p_n$ are stated in \cite[p. 293]{Comtet} and a nice exercise.
%These integrals may be evaluated in terms of Stirling cycle numbers with \e{dog}.

We may set $\mathcal{C}^-_n := n! e_n$ and $\mathcal{C}^+_n:= n! p_n$ where the signs indicate the falling and rising factorials in \e{sgns}. Among other names, $\mathcal{C}^-_n$ and $\mathcal{C}^+_n$ are known as  {\em Cauchy numbers}; \cite{Bla} has more information about them and their history. The  transition formulas give the simple identities
\begin{equation}\label{sib}
  \mathcal{C}^-_n =n! \sum_{k=0}^n (-1)^{n-k} \dm_{n,k}\left( \frac{1}{2}, \frac{1}{3}, \dots \right), \qquad
  \mathcal{C}^+_n =n! \sum_{k=1}^n \frac{(-1)^{k-1}}{k} \dm_{n,k}\left( \frac{1}{2}, \frac{1}{3}, \dots \right),
\end{equation}
which may also be written with Stirling cycle numbers using \e{jf2}. We will later see the N\"{o}rlund polynomial expressions
\begin{equation}\label{logb}
  \mathcal{C}^-_n = -\frac{1}{n-1} B_n^{(n-1)}, \qquad \mathcal{C}^+_n = (-1)^n  B_n^{(n)}.
\end{equation}
\end{proof}

\begin{eg}[Tangent, inverse]
$$
\left( \frac{t}{\arctan t},  \ \frac{\arctan t}{t},  \ -\frac{1}{t}+\frac{1}{(1+t^2)\arctan t}\right),
$$
after applying $*$ to Example \ref{tg}, with $h_{2n}=(-1)^n/(2n+1)$ and $h_{2n-1}=0$.
\end{eg}

\begin{eg}[Hermite polynomial, inverse] For $n\gqs 1$,
%\begin{equation}\label{ih}
%  e_n = -(n-1)^{n/2-1} \frac{\mathcal{H}_n \left(i \sqrt{n-1} x\right)}{i^n n!}
%  \quad
%  h_n = (n+1)^{n/2-1} \frac{\mathcal{H}_n \left(i \sqrt{n+1} x\right)}{(-i)^n n!}
%  \quad
%  p_n = n^{n/2} \frac{\mathcal{H}_n \left(i \sqrt{n} x\right)}{(-i)^n n!}.
%\end{equation}
\begin{equation}\label{ihh}
  e_n = - \frac{\mathcal{H}_n \left(i x \sqrt{n-1} \right)}{(n-1)^{1-n/2} i^n n!},
  \quad
  h_n =  \frac{\mathcal{H}_n \left(i x \sqrt{n+1} \right)}{(n+1)^{1-n/2}(-i)^n n!},
  \quad
  p_n =  \frac{\mathcal{H}_n \left(i x \sqrt{n} \right)}{n^{-n/2}(-i)^n n!}.
\end{equation}
\end{eg}
\begin{proof}[Discussion]
(To find $e_1$ use $h_1$ instead, as $e_1=h_1$.) This is the result of applying $*$ to  Example \ref{herm}. Letting $t\to i\sqrt{\alpha} t$ and $x\to i\sqrt{\alpha} x$ for $\alpha \gqs 0$ means
\begin{equation*}
  e^{2xt - t^2} = \sum_{n=0}^\infty \frac{\mathcal{H}_n(x)}{n!} t^n \implies [t^n] e^{-\alpha (2xt -  t^2)} = (i\sqrt{\alpha})^n \frac{\mathcal{H}_n(i x \sqrt{\alpha})}{n!},
\end{equation*}
and we can explicitly compute \e{rem} for $H(t)=e^{2xt - t^2}$ to give \e{ihh}.
\end{proof}

\section{Combining powers and inverses of series}

\subsection{General series}

The results in this section were inspired by Example 25 of \cite{Macd} as well as the generalized series in \cite[Sect. 5.4]{Knu}.

\begin{adef} \label{lop}
{\rm Let $\rho(t)=1+\rho_1 t+ \rho_2 t^2+ \cdots$ be a power series  and write $D_n(\alpha):=[t^n]\rho(t)^\alpha$. For all $\alpha$, define the {\em general series associated to} $\rho$ as
\begin{equation}\label{lmb}
  \mathcal{Q}_{\rho,\alpha}(t) = \mathcal{Q}_\alpha(t) := 1+\sum_{n=1}^\infty \frac{ D_n(\alpha n + 1)}{\alpha n + 1} t^n.
\end{equation}
}
\end{adef}

Clearly $\mathcal{Q}_0(t) = \rho(t)$. Note that $\rho(t)$ and $\mathcal{Q}_\alpha(t)$ are formal power series for now and we may not have convergence for any $t$.  The denominator $\alpha n + 1$ can be zero in \e{lmb} without causing an issue: for $n\gqs 1$,
\begin{equation} \label{zr}
  D_n(\beta)=\sum_{k=1}^n \binom{\beta}{k} \dm_{n,k}(\rho_1, \rho_2, \dots) = \beta \sum_{k=1}^n \frac 1k \binom{\beta-1}{k-1} \dm_{n,k}(\rho_1, \rho_2, \dots)
\end{equation}
so that $D_n(\beta)/\beta$ is  well-defined for all $\beta$ when $n\gqs 1$.

%We focus on this construction because it has such simple expansion coefficients. The next results describe the interplay between the $*$ operation and powers.

Some remarkable properties of these general series are given next. Recall the $*$ operation from \e{star}, related to the compositional inverse.

\begin{theorem} \label{general}
Let $\mathcal{Q}_\alpha(t)=\mathcal{Q}_{\rho,\alpha}(t)$ be the general series associated to  a power series  $\rho$ with constant term $1$. Then  it satisfies
\begin{equation}\label{eqsat}
  \mathcal{Q}_\alpha(t)  =  \rho\left(t \mathcal{Q}_\alpha(t)^\alpha \right).
\end{equation}
For  all  $\alpha$ and $\beta$ we have
\begin{align} \label{cc}
  \mathcal{Q}_\alpha(t)^\beta & = 1+\sum_{n=1}^\infty \frac{\beta}{\alpha n + \beta} D_n(\alpha n + \beta) \cdot t^n,
  \\
  \mathcal{Q}_\alpha(t)^\beta
  \left( 1+\alpha t \frac{\mathcal{Q}'_\alpha(t)}{\mathcal{Q}_\alpha(t)}\right)
  & = 1+\sum_{n=1}^\infty D_n(\alpha n + \beta) \cdot  t^n, \label{cc2}
  \\
  \left( \mathcal{Q}_\alpha(t)^\beta \right)^* & = \mathcal{Q}_{\alpha-\beta}(t)^{-\beta}. \label{invst}
\end{align}
\end{theorem}

\begin{cor}
Let $\rho(t)$ be a power series  with constant term $1$. The two-parameter family $\mathcal{Q}_{\rho,\alpha}(t)^\beta$ contains all possible series obtained from $\rho(t)$ by repeatedly taking powers  and applying the $*$ operator.
\end{cor}

For example, with $\mathcal{Q}_{\alpha}=\mathcal{Q}_{\rho,\alpha}$,
\begin{equation}\label{muk}
  \rho= \mathcal{Q}_{0}, \qquad \rho^\beta = \mathcal{Q}_{0}^\beta, \qquad \rho^* = \mathcal{Q}_{-1}^{-1},
\qquad (\rho^\beta)^* = \mathcal{Q}_{-\beta}^{-\beta},  \qquad (((\rho^\beta)^*)^\alpha)^* = \mathcal{Q}_{\beta(\alpha-1)}^{\beta \alpha}.
\end{equation}

\begin{cor}\label{newtrip}
Let $\rho(t)$ be a power series  with %constant term $1$ and
associated general series $\mathcal{Q}_\alpha(t)$. Then
\begin{equation*}
  \left(\mathcal{Q}_\alpha(-t)^{-\beta}, \ \mathcal{Q}_\alpha(t)^\beta, \ \beta \mathcal{Q}'_\alpha(t)/\mathcal{Q}_\alpha(t) \right)
\end{equation*}
is a symmetric triple, and for $n\gqs 1$,
\begin{equation} \label{hu}
  e_n= \frac{(-1)^n(-\beta)}{\alpha n - \beta} D_n(\alpha n - \beta), \quad h_n = \frac{\beta}{\alpha n + \beta} D_n(\alpha n + \beta), \quad p_n=\frac{\beta}{\alpha} D_n(\alpha n).
\end{equation}
\end{cor}
\begin{proof}
The formulas in \e{hu} come from \e{cc} and \e{cc2}.
\end{proof}

\subsection{Lagrange inversion and the proof of Theorem \ref{general}}

\begin{theorem}[Lagrange inversion]
Let $F$ be a formal power series with  compositional inverse $G$ so that $G(x)=t \iff F(t)=x$. Then for all integers $n$ and all formal Laurent series $\phi$,  we have
\begin{equation}\label{lin}
  [x^n] \phi(G(x)) = [t^{-1}] \frac{\phi(t) F'(t)}{F(t)^{n+1}}.
\end{equation}
\end{theorem}
This is proved in Theorem 2.1.1 and (2.1.7) of \cite{Ge16}. As explained there, we find another useful form of \e{lin} by setting
$\psi(t):=\phi(t) t F'(t)/F(t)$. Then
\begin{equation*}
  \psi(G(x))=\phi(G(x)) G(x) \frac{F'(G(x))}{F(G(x))} = \phi(G(x))  \frac{G(x)}{x G'(x)}
\end{equation*}
and by \e{lin},
\begin{equation*}
  [x^n] \psi(G(x))\frac{G'(x)}{G(x)} = [x^{n+1}] \phi(G(x)) = [t^{-1}] \frac{\phi(t) F'(t)}{F(t)^{n+2}}=[t^{0}] \frac{\psi(t)}{F(t)^{n+1}}.
\end{equation*}

Writing $F(t)=t f(t)$, where $f$ necessarily has a nonzero constant term, gives the   versions we will need:

\begin{cor} \label{licor}
Let $F$ be a formal power series with  compositional inverse $G$. Write $F(t)=t f(t)$. Then for all integers $n$ and all formal Laurent series $\phi$ and $\psi$ we have
\begin{equation}\label{lin2}
  [x^n] \phi(G(x)) = [t^{n}] \left( 1+t \frac{f'(t)}{f(t)}\right)\frac{\phi(t)}{f(t)^{n}}, \qquad [x^n] \psi(G(x)) \frac{G'(x)}{G(x)} = [t^{n+1}] \frac{\psi(t)}{f(t)^{n+1}}.
\end{equation}
\end{cor}

\begin{proof}[Proof of Theorem \ref{general}]
We will use Corollary \ref{licor}, and in its notation write $f(t):=\rho(t)^{-\alpha}$ with $G$ the compositional inverse of $F(t)=t f(t)$. Therefore $F(G(t))=t$ implies
\begin{equation} \label{wa}
  G(t) \cdot \rho(G(t))^{-\alpha} =t.
\end{equation}
Set $\mathcal{Q}_\alpha(t)$ to be $\rho(G(t))$ and we will see that this agrees with Definition \ref{lop}. First note that \e{wa} implies \e{eqsat}.
Applying Corollary \ref{licor} with $\phi(t)=\rho(t)^{\beta}$ and $f(x)=\rho(x)^{-\alpha}$ gives
\begin{equation*}
  [t^n] \mathcal{Q}_\alpha(t)^\beta = [t^{n}] \left( 1+t \frac{f'(t)}{f(t)}\right)\frac{\phi(t)}{f(t)^{n}}
  = [t^{n}] \left( \rho(t)^{\alpha n+\beta} - \alpha t \rho'(t) \rho(t)^{\alpha n+\beta-1}\right).
\end{equation*}
This equals
\begin{multline*}
   [t^{n}] \left( \left(1+ \frac{\alpha}{\alpha n+\beta} \right) \rho(t)^{\alpha n+\beta} - \frac{\alpha}{\alpha n+\beta} \frac{d}{dt} \left( t  \rho(t)^{\alpha n+\beta}\right)\right)\\
   = \left(1+ \frac{\alpha}{\alpha n+\beta} \right) D_n(\alpha n+\beta)
    - \frac{\alpha (n+1)}{\alpha n+\beta}  D_n(\alpha n+\beta) =
   \frac{\beta}{\alpha n + \beta} D_n(\alpha n + \beta).
\end{multline*}
Therefore \e{cc} is true and for $\beta=1$ we see that $\mathcal{Q}_\alpha(t)$ matches our original definition in \e{lmb}.

From \e{wa},
\begin{equation*}
   G(t) = t \mathcal{Q}_\alpha(t)^{\alpha}
  \qquad \text{and hence} \qquad t\frac{G'(t)}{G(t)} = 1+\alpha t \frac{\mathcal{Q}'_\alpha(t)}{\mathcal{Q}_\alpha(t)}.
\end{equation*}
Therefore
\begin{equation*}
  [t^n]\mathcal{Q}_\alpha(t)^\beta
  \left( 1+\alpha t \frac{\mathcal{Q}'_\alpha(t)}{\mathcal{Q}_\alpha(t)}\right)
  =[t^n] \mathcal{Q}_\alpha(t)^\beta t \frac{G'(t)}{G(t)} = [t^{n-1}] \rho(G(t))^\beta \frac{G'(t)}{G(t)}
\end{equation*}
and applying Corollary \ref{licor} with $\psi(t)=\rho(t)^{\beta}$ and $f(x)=\rho(x)^{-\alpha}$ gives \e{cc2} directly.

To prove \e{invst}, write $G_u(t)$ for the compositional inverse of $F_u(t):=t \rho(t)^u$. As in \e{wa},
\begin{equation} \label{wa3}
  G_u(t) \cdot \rho(G_u(t))^{u} =t.
\end{equation}
From the definition of $\mathcal{Q}_\alpha(t)$ and  \e{wa3},
\begin{equation} \label{wa8}
  t \mathcal{Q}_\alpha(t)^\beta = t \rho(G_{-\alpha}(t))^\beta = G_{-\alpha}(t) \cdot \rho(G_{-\alpha}(t))^{-\alpha} \cdot \rho(G_{-\alpha}(t))^\beta = F_{\beta-\alpha}(G_{-\alpha}(t)).
\end{equation}
For \e{invst} we need the compositional inverse of \e{wa8}:
\begin{multline*}
  t\left( \mathcal{Q}_\alpha(t)^\beta \right)^* = F_{-\alpha}(G_{\beta-\alpha}(t))
  = G_{\beta-\alpha}(t) \rho(G_{\beta-\alpha}(t))^{-\alpha} \\
  = t \rho(G_{\beta-\alpha}(t))^{\alpha-\beta}\rho(G_{\beta-\alpha}(t))^{-\alpha}
  = t \rho(G_{\beta-\alpha}(t))^{-\beta}.
\end{multline*}
Comparing this with the leftmost equality in \e{wa8} completes the proof of \e{invst}.
\end{proof}

Note that by \e{wa} and the definition of $G$, the general series associated to $\rho$ satisfies
\begin{equation*}
  \mathcal{Q}_\alpha(t) = \rho\left( (\rho(t)^{-\alpha})^*\right) = \left( (\rho(t)^{-\alpha})^*\right)^{1/\alpha}.
\end{equation*}

\section{Further examples of symmetric triples} \label{fur}

\subsection{Binomial, exponential and Bernoulli generalizations} \label{beb}
The simplest example of Theorem \ref{general} (besides $\rho(t)=1$) is $\rho(t)=1+t$. In that case $D_n(\alpha)=\binom{\alpha}{n}$.  The series $\mathcal{B}_\alpha(t):=\mathcal{Q}_{\rho,\alpha}(t)$ is called a generalized binomial series in \cite[Eq. (5.60)]{Knu}. Its properties   were first developed by Lambert in the 1750s and Theorem \ref{general} tells us it satisfies
\begin{align}
   \mathcal{B}_\alpha(t)  & =  1 + t \mathcal{B}_\alpha(t)^\alpha, \label{cats}\\
   \mathcal{B}_\alpha(t)^\beta & = 1+\sum_{n=1}^\infty \frac{\beta}{\alpha n + \beta} \binom{\alpha n + \beta}{n} t^n. \label{cats2}
\end{align}
Then Corollary \ref{newtrip} provides the corresponding symmetric triple, generalizing Example \ref{ex-bc}:

\begin{eg}[General binomial] \label{ex-bci2} %ized Fuss-Catalan numbers]
\begin{equation}\label{inn2}
  e_n = \frac{\beta}{(1-\alpha)n + \beta}\binom{(1-\alpha)n + \beta}{n}, \qquad h_n = \frac{\beta}{\alpha n + \beta}\binom{\alpha n + \beta}{n},  \qquad p_n  = \frac{\beta}{\alpha}\binom{\alpha n}{n}.
\end{equation}
\end{eg}

This is \cite[Ex. 25(a)]{Macd}. The $\alpha$ and $\beta$ in \e{inn2} are arbitrary, and specializing them to $\alpha=2$, $\beta=1$, for example, gives a symmetric triple for the Catalan numbers $C_n:=\binom{2n}{n}/(n+1)$.

\begin{eg}[Catalan numbers] \label{cata}
For $n\gqs 1$, %\label{ex-bci2} %ized Fuss-Catalan numbers]
\begin{equation}\label{cat}
  e_n = (-1)^{n-1}C_{n-1}, \qquad h_n = C_n,  \qquad p_n  = \frac{n+1}{2}C_n.
\end{equation}
\end{eg}

\begin{proof}[Discussion] The $h_n \leftrightarrow p_n$ transition formula implies the recursion
\begin{equation}\label{crec}
  C_n =\sum_{k=0}^n \frac{1}{2^k k!} \dm_{n,k}\left( \frac{2C_1}1, \frac{3C_2}{2}, \frac{4C_3}3, \frac{5C_4}{4}, \dots \right).
\end{equation}
Solving \e{cats} in this case gives the familiar generating function
\begin{equation}\label{cgf}
  \mathcal{B}_2(t) = \frac{1-\sqrt{1-4t}}{2t} = \sum_{n=0}^\infty C_n t^n.
\end{equation}
See \cite[p. 203]{Knu}, \cite[\S\S 2.3, 3.3, 3.4]{Ge16} for many closely related series.
\end{proof}

The next simplest case has $\rho(t)=e^t$ and $D_n(\alpha)=\alpha^n/n!$. Following \cite[Eq. (5.60)]{Knu} again,  $\mathcal{E}_\alpha(t):=\mathcal{Q}_{\rho,\alpha}(t)$ is called a generalized exponential series.  Theorem \ref{general} shows
\begin{align} \label{ew}
   \mathcal{E}_\alpha(t)  & =  \exp\left( t \mathcal{E}_\alpha(t)^\alpha\right), \\
   \mathcal{E}_\alpha(t)^\beta & = 1+\sum_{n=1}^\infty \beta\frac{(\alpha n+\beta)^{n-1}}{n!} t^n. \label{ew2}
\end{align}
The Lambert $W$ function $W(t)$ and the tree function $T(t)$ (see \cite[Sect. 3.2]{Ge16}) can be recognized here from \e{ew}:
\begin{equation}\label{w-t}
  W(t) e^{W(t)}=t, \quad T(t) e^{-T(t)}=t \qquad \implies \qquad W(t)=t \mathcal{E}_1(-t), \qquad T(t)=t \mathcal{E}_1(t).
\end{equation}

Corollary \ref{newtrip} gives the following symmetric triple, generalizing Example \ref{exp}.

\begin{eg}[General exponential] \label{ex-epi2}
\begin{equation}\label{inex2}
  e_n = \beta\frac{(-\alpha n+\beta)^{n-1}}{n!}, \qquad h_n = \beta\frac{(\alpha n+\beta)^{n-1}}{n!},  \qquad p_n  = \beta\frac{(\alpha n)^{n-1}}{(n-1)!}.
\end{equation}
\end{eg}

\begin{proof}[Discussion]
This is \cite[Exs. 14, 25(b)]{Macd}. Special cases of \e{inex2} are, (see \e{muk}),
\begin{align*}
  (e^t,e^t,1) & \quad \text{when} \quad \alpha=0,  \beta=1, \\
  (e^{x t},e^{x t}, x ) & \quad \text{when} \quad \alpha=0,  \beta =x,\\
  (e^{x t},e^{x t}, x )^* & \quad \text{when} \quad  \alpha=-x,   \beta =-x.
\end{align*}
The convolution identity \e{peh} gives
\begin{equation}\label{peh2}
  \bigl(\alpha (n+1)\bigr)^{n} = \sum_{j=0}^{n} \binom{n}{j} \beta \bigl(\beta+\alpha j \bigr)^{j-1} \bigl(\alpha(n+1)-\beta-\alpha j\bigr)^{n-j}
\end{equation}
here, which is a special case %(though quite a lot of)
of Abel's 1826 generalization of the binomial theorem \cite[p. 128]{Comtet}.
 \end{proof}

Taking $\rho(t)=t/(e^t-1)$ in Theorem \ref{general} has  $D_n(\alpha)=B^{(\alpha)}_n/n!$, recalling the N\"{o}rlund polynomials $B_n^{(z)}$ with generating function
\begin{equation}\label{stpo}
  \left(\frac{t}{e^{t}-1} \right)^{z} = \sum_{n=0}^\infty B_n^{(z)} \frac{t^n}{n!}.
\end{equation}
We may describe $\mathcal{U}_\alpha(t):=\mathcal{Q}_{\rho,\alpha}(t)$ as a generalized Bernoulli series, and Theorem \ref{general} implies
\begin{align}
   e^{t \cdot \mathcal{U}_\alpha(t)^\alpha}  & =  1+ t \cdot  \mathcal{U}_\alpha(t)^{\alpha-1}, \label{ud}\\
   \mathcal{U}_\alpha(t)^\beta & = 1+\sum_{n=1}^\infty \beta \frac{B^{(\alpha n + \beta)}_n}{\alpha n + \beta} \frac{t^n}{n!}. \label{ud2}
\end{align}
Then $\mathcal{U}_\alpha(t)$ is  related to the series $\mathcal{S}_\alpha(t)$ in \cite[p. 272]{Knu} by
$\mathcal{U}_\alpha(t)=\mathcal{S}_\alpha(-t)$.
The associated symmetric triple from Corollary \ref{newtrip} generalizes Examples \ref{nou}, \ref{logx} involving Bernoulli numbers and logarithms:

\begin{eg}[General Bernoulli] \label{ex-gber}
\begin{equation}\label{gber2}
  e_n= (-1)^n\frac{(-\beta)}{n!} \frac{B_n^{(\alpha n - \beta)}}{\alpha n - \beta}, \qquad h_n = \frac{\beta}{n!} \frac{B_n^{(\alpha n + \beta)}}{\alpha n + \beta}, \qquad p_n=\frac{\beta}{n!} \frac{B_n^{(\alpha n)}}{\alpha}.
\end{equation}
\end{eg}

The general series  $\mathcal{B}_\alpha(t)$, $\mathcal{E}_\alpha(t)$ and $\mathcal{U}_\alpha(t)$ have been treated formally so far. We show next that they converge for small enough $t$.

\begin{prop}
Let $\alpha$, $\beta$ and $t$ be any complex numbers.  The series \e{cats2} and \e{ud2} for $\mathcal{B}_\alpha(t)^\beta$ and $\mathcal{U}_\alpha(t)^\beta$ converge absolutely for $|t|< 2^{-|\alpha|-1}$. The series \e{ew2} for $\mathcal{E}_\alpha(t)^\beta$ converges absolutely for $|t|< 1/(e|\alpha|)$.
\end{prop}
\begin{proof}
Take $r$ and $\varepsilon$ to be real numbers with $r\gqs 0$ and $|\varepsilon|<1$. Then
\begin{equation*}
  (1-\varepsilon)^{-r} = \sum_{k=0}^\infty \binom{-r}{k} (-\varepsilon)^k = \sum_{k=0}^\infty \binom{k+r-1}{k} \varepsilon^k
\end{equation*}
and $\binom{k+r-1}{k} \gqs 0$. Setting $\varepsilon =1/2$ shows
\begin{equation}\label{bt}
  \sum_{k=0}^n \binom{k+r-1}{k} \frac 1{2^k} \lqs 2^r \qquad \text{and} \qquad \binom{n+r-1}{n}  \lqs 2^{r+n}.
\end{equation}
Since we easily have $|\binom{z}{k}| \lqs \binom{k+|z|-1}{k}$ for all $z\in \C$, it follows from the right bound in \e{bt} that
\begin{equation*}
  \left| \frac{\beta}{\alpha n + \beta} \binom{\alpha n + \beta}{n} \right| \ll 2^{n+|\alpha n + \beta|}
\end{equation*}
and this gives the desired domain of convergence for $\mathcal{B}_\alpha(t)^\beta$.

Looking ahead to \e{nr2}, we may apply the bound $|\dm_{n,k}\left(1/2!,1/3!,\dots \right)| \lqs 2^{n-k}$ from Lemma 2.2 of \cite{odm}, along with the left bound in \e{bt}, to see that
\begin{equation*}
   \left| \frac{B^{(\alpha n + \beta)}_n}{\alpha n + \beta} \frac{\beta}{n!} \right| \ll
   \sum_{k=0}^n \binom{k+|\alpha n + \beta|-1}{k} 2^{n-k} \lqs 2^{n+|\alpha n + \beta|}.
\end{equation*}

Lastly, a routine application of Stirling's formula gives the convergence for $\mathcal{E}_\alpha(t)^\beta$.
\end{proof}

The radius of convergence we gave for $\mathcal{E}_\alpha(t)^\beta$ is  exact by the ratio test (with convergence in all of $\C$ when $\alpha=0$). It would be interesting to find the exact radii of convergence in the other two cases, and if the functions can be continued past that.

General series associated to the partition and representation series in \e{genpart}, \e{rpw} and \e{trp} may also be constructed, along with their symmetric triples.

\subsection{N\"{o}rlund polynomials} \label{norl}
To get a better understanding of Example \ref{ex-gber} we next look at the N\"{o}rlund polynomials in  detail.  Expanding $\rho(t)$ and its reciprocal in \e{stpo} and then applying \e{pot} finds
\begin{align}
  \frac{B_n^{(z)}}{n!} & =  \sum_{k=0}^n \binom{z}{k} \dm_{n,k}\left(\frac{B_1}{1!},\frac{B_2}{2!},\frac{B_3}{3!},\dots \right) \label{nr} \\
 & =  \sum_{k=0}^n \binom{-z}{k} \dm_{n,k}\left(\frac{1}{2!},\frac{1}{3!},\frac{1}{4!},\dots \right), \label{nr2}
\end{align}
making $B_n^{(z)}$  a polynomial in $z$ of degree $n$ with rational coefficients. For small $n$:
\begin{table}[ht]
\centering
%{\footnotesize
\begin{tabular}{c|c|c|c|c|c}
%\hline
 $n$   & $0$ & 1 & 2 & 3 & 4  \\
\hline
 $(-2)^n B_n^{(z)}$ & $1$ &  $z$ & $z^2-\frac{z}{3}$ & $z^3-z^2$ & $z^4-2 z^3+\frac{z^2}{3}+\frac{2
   z}{15}$\\
 %\hline
\end{tabular}%}
%\caption{The approximations of Theorem \ref{mainthm3} to $\g(1000)$.} \label{xeb}
\end{table}

\noindent
They are closely related to the Stirling polynomials $\sigma_n(z)$ of \cite[Eq. (6.45)]{Knu} with $(-1)^n B_n^{(z)} = n! z \sigma_n(z)$.
The recursion
\begin{equation}\label{brec}
  z B_n^{(z+1)} =(z-n)B_n^{(z)} -z n B_{n-1}^{(z)},
\end{equation}
as in \cite[p. 329]{chacha},  comes from differentiating \e{stpo}.

 The next formula for $B_n^{(z)}$ is similar to \e{nr} and \e{nr2} but relates  to the logarithm. For this, note that
\begin{align}
  \mathcal{U}_0(t) = \rho(t) = \frac{t}{e^{t}-1} & \implies  \mathcal{U}_0(t)^{-1}= \frac{e^{t}-1}{t}\notag\\
  & \implies  (\mathcal{U}_0(t)^{-1})^* = \mathcal{U}_1(t) = \frac{\log(1+t)}{t} \label{po}\\
  & \implies  \mathcal{U}_1(t)^{\beta} = \left(\frac{\log(1+t)}{t} \right)^{\beta}. \notag
\end{align}
Hence \e{ud2} provides
\begin{equation}\label{ud3}
  \left(\frac{\log(1+t)}{t} \right)^{z} = \sum_{n=0}^\infty \frac{z}{n+z} B_n^{(n+z)} \frac{t^n}{n!}.
\end{equation}
and then, expanding the left side with \e{pot},
\begin{equation}\label{nr3}
  \frac{B_n^{(n+z)}}{n!}  = (-1)^n \frac{n+z}{z} \sum_{k=0}^n \binom{z}{k} \dm_{n,k}\left(\frac{1}{2},\frac{1}{3},\frac{1}{4},\dots \right).
\end{equation}
We also see from \e{po} that the logarithmic triple in Example \ref{logx} is just the case $\alpha=\beta=1$ of Example \ref{gber2} with $t\to -t$. The formulas \e{logb} for the Cauchy numbers follow.

The N\"{o}rlund polynomials are also related to the Stirling numbers. For integers $m$, $k \gqs 0$
\begin{equation} \label{nut}
  \stirb{m+k}{k} = \binom{m+k}{m} B_{m}^{(-k)}, \qquad
   \stira{m+k}{k} =  \binom{-k}{m} B_{m}^{(m+k)},
\end{equation}
where the left identity comes from comparing \e{bakb} and \e{stpo}, while the right one comes from comparing \e{baka} and \e{ud3}.
The right sides of both identities in \e{nut} are degree $2m$  polynomials in $k$, allowing us to extend the Stirling number definitions. % it follows that the Stirling numbers make sense whenever the top number is an integer $m\gqs 0$ larger than the bottom number.
Inserting our formulas \e{nr}, \e{nr2} and \e{nr3} into \e{nut}  recovers Kramp's  identities \e{kr} and \e{kr2} in two cases.

Now $B_{m}^{(k)}$ may be given explicitly in  different ranges of  $k\in \Z$. For $k$ close to $m$ we have
\begin{gather}\label{bnz}
    B_m^{(m-1)} = -(m-1)\mathcal{C}^-_m,  \qquad   B_m^{(m)} = (-1)^m \mathcal{C}^+_m, \\
 \frac{B_m^{(m+1)}}{m!} = (-1)^m,  \qquad   \frac{B_m^{(m+2)}}{m!} = (-1)^m H_{m+1},  \qquad   \frac{B_m^{(m+3)}}{m!} = H_{m+2}^2 - H_{m+2}^{(2)}, \label{bnzz}
\end{gather}
by \e{logb}, \e{nut} and the Stirling cycle number formulas \e{harm1}, \e{bg3} and \e{bg4}. For $k$ close to $0$ we have
\begin{gather}\label{bnz2}
     B_m^{(1)} = B_m,  \qquad   B_m^{(2)} = (1-m)B_m -m B_{m-1}, \\
 B_m^{(0)} = \delta_{m,0},  \qquad   \frac{B_m^{(-1)}}{m!} = \frac 1{(m+1)!},  \qquad   \frac{B_m^{(-2)}}{m!} = \frac{2^{m+2}-2}{(m+2)!}, \label{bnzz2}
\end{gather}
with  \e{bnzz2} following from \e{nut} and the Stirling subset number formula \e{ssub}. The first identity in \e{bnz2} is clear from \e{stpo}; apply the recursion \e{brec} for the second.

The following two examples are cases of Example \ref{ex-gber}. The first has $\alpha=0$, $\beta=2$ corresponding to the square of Example \ref{nou}. The second has $\alpha=1$, $\beta=2$ corresponding to the square of Example \ref{logx}.

\begin{eg}[Bernoulli numbers, squared] %\label{nou}
\begin{equation}\label{xber2w}
  e_n= (-1)^n \frac{2^{n+2}-2}{(n+2)!}, \qquad  h_n= \frac{(1-n)B_n -n B_{n-1}}{n!}, \qquad  p_k= (-1)^{n-1}\frac{2 B_n}{n!}.
\end{equation}
\end{eg}

\begin{eg}[Logarithms, squared] \label{logxs}
\begin{equation*}
  e_n=(-1)^{n-1}\frac{2}{n-2}\frac{B_n^{(n-2)}}{n!} , \qquad h_n=(-1)^n \frac{2}{n+2} H_{n+1}, \qquad p_n=\frac{2\mathcal{C}^+_n}{n!}.
\end{equation*}
\end{eg}

\subsection{Special values of De Moivre polynomials}

It is useful to be able to simplify De Moivre polynomials and, in particular,
know when they can be evaluated explicitly. Clearly
\begin{equation}\label{expl}
  \dm_{m+k,k}(a_0,a_1,a_2,\dots) = b_m  \iff [x^m]\left(a_0+a_1 x + a_2 x^2+ \cdots \right)^k = b_m,
\end{equation}
for $m,k \gqs 0$, and so we are looking for power series whose powers are known. Simple examples are
\begin{equation*}
  ((1+x)^\beta)^k = (1+x)^{\beta k}, \qquad (e^{\beta x})^k = e^{\beta k x}
\end{equation*}
for $\beta$ any complex number or indeterminate, giving
\begin{equation} \label{idb}
  \dm_{m+k,k}\left(\binom{\beta}{0},\binom{\beta}{1},\binom{\beta}{2},\dots \right) = \binom{\beta k}{m}, \qquad
  \dm_{m+k,k}\left(\frac{\beta^0}{0!},\frac{\beta^1}{1!},\frac{\beta^2}{2!},\dots \right) = \frac{(\beta k)^m}{m!}.
\end{equation}

Our work in sections \ref{beb}, \ref{norl} shows more possibilities.
From the generalized exponential series $\mathcal{E}_\alpha(t)$ in \e{ew2},
\begin{equation}\label{ewg}
   \dm_{m+k,k}\left(1,\frac{\beta}{1!},(2\alpha+\beta)\frac{\beta}{2!},(3\alpha+\beta)^2\frac{\beta}{3!},
   (4\alpha+\beta)^3\frac{\beta}{4!},\dots \right) = (m\alpha+k \beta)^{m-1} \frac{k\beta}{m!},
\end{equation}
for arbitrary $\alpha$ and $\beta$.
The case $\alpha=\beta=1$  is
\begin{equation}\label{ewg2}
   \dm_{m+k,k}\left(\frac{1^0}{1!},\frac{2^1}{2!},\frac{3^2}{3!},
   \frac{4^3}{4!},\dots \right) = (m+k)^{m-1} \frac{k}{m!},
\end{equation}
valid for $m$, $k\gqs 0$ and not both $0$. This is equivalent to an identity on p. 89 of \cite{Chu19}.

From the generalized binomial series $\mathcal{B}_\alpha(t)$ in \e{cats2},
\begin{multline}\label{gb}
   \dm_{m+k,k}\left(1,\frac{\beta}{\alpha+\beta}\binom{\alpha+\beta}{1},
   \frac{\beta}{2\alpha+\beta}\binom{2\alpha+\beta}{2},\frac{\beta}{3\alpha+\beta}\binom{3\alpha+\beta}{3},\dots \right) \\
   = \frac{\beta k}{\alpha m+\beta k}\binom{\alpha m+\beta k}{m}.
\end{multline}
As we saw in Example \ref{cata}, the case $\alpha=2$, $\beta=1$ of this gives the Catalan numbers, so that
\begin{equation}\label{gb2}
  \dm_{m+k,k}\left(C_0, C_1, C_2,  \dots \right)  = \frac{k}{2m+k}\binom{2m+k}{m}.
\end{equation}
This identity \e{gb2} is in fact equivalent to Theorem  1.2 of \cite{Qi17} where they also look at variations and applications.

The generalized Bernoulli series $\mathcal{U}_\alpha(t)$ in \e{ud2} provides
\begin{equation}\label{ug}
   \dm_{m+k,k}\left(1,\frac{B_1^{(\alpha+\beta)}}{\alpha+\beta}\frac{\beta}{1!},
   \frac{B_2^{(2\alpha+\beta)}}{2\alpha+\beta}\frac{\beta}{2!},
   \frac{B_3^{(3\alpha+\beta)}}{3\alpha+\beta}\frac{\beta}{3!},
   ,\dots \right) = \frac{B_m^{(\alpha m +\beta k)}}{\alpha m+\beta k}\frac{\beta k}{m!}.
\end{equation}
The evaluations in section \ref{norl} and \e{bnz} -- \e{bnzz2} give many explicit cases of this. For example, with $\alpha=0$ and $\beta=1,-1,-2$,
\begin{align}
  \dm_{m+k,k}\left(1,\frac{B_1}{1!},
   \frac{B_2}{2!},
   \frac{B_3}{3!}
   ,\dots \right)  & = \frac{B_m^{(k)}}{m!},
   \\
   \dm_{m+k,k}\left(\frac{1}{1!},
   \frac{1}{2!},
   \frac{1}{3!}
   ,\dots \right) &  = \frac{k!}{(m+k)!} \stirb{m+k}{k}, \label{fut2}
   \\
  \dm_{m+k,k}\left(\frac{2^2-2}{2!}, \frac{2^3-2}{3!}, \frac{2^4-2}{4!},
   ,\dots \right)  & = \frac{(2k)!}{(m+2k)!} \stirb{m+2k}{2k}.
\end{align}
For $\alpha=1$ and $\beta=1,2$,
\begin{align}
  \dm_{m+k,k}\left(\frac{1}{1},
   \frac{1}{2},
   \frac{1}{3}
   ,\dots \right) &  = \frac{k!}{(m+k)!} \stira{m+k}{k}, \label{fut}
   \\
   \dm_{m+k,k}\left(\frac{H_1}{2}, \frac{H_2}{3}, \frac{H_3}{4},  \cdots \right) & =  \frac{(2k)!}{2^k (m+2k)!} \stira{m+2k}{2k}.
\end{align}
The series \e{genpart}, \e{rpw}  and \e{trp} may also be used to find similar identities.

{\small \bibliography{dem-bib} }

\begin{thebibliography}{QSLK17}

\bibitem[AAR99]{AAR}
George~E. Andrews, Richard Askey, and Ranjan Roy.
\newblock {\em Special functions}, volume~71 of {\em Encyclopedia of
  Mathematics and its Applications}.
\newblock Cambridge University Press, Cambridge, 1999.

\bibitem[And98]{And76}
George~E. Andrews.
\newblock {\em The theory of partitions}.
\newblock Cambridge Mathematical Library. Cambridge University Press,
  Cambridge, 1998.
\newblock Reprint of the 1976 original.

\bibitem[AS64]{AS}
Milton Abramowitz and Irene~A. Stegun.
\newblock {\em Handbook of mathematical functions with formulas, graphs, and
  mathematical tables}.
\newblock National Bureau of Standards Applied Mathematics Series, No. 55. U.
  S. Government Printing Office, Washington, D. C., 1964.
\newblock %For sale by the Superintendent of Documents.

\bibitem[Bat17]{Bat}
Necdet Batir.
\newblock On some combinatorial identities and harmonic sums.
\newblock {\em Int. J. Number Theory}, 13(7):1695--1709, 2017.

\bibitem[BBT09]{bo}
Sadek Bouroubi and Nesrine Benyahia~Tani.
\newblock A new identity for complete {B}ell polynomials based on a formula of
  {R}amanujan.
\newblock {\em J. Integer Seq.}, 12(3):Article 09.3.5, 6, 2009.

\bibitem[Bel28]{Bell27}
Eric~T. Bell.
\newblock Partition polynomials.
\newblock {\em Ann. of Math. (2)}, 29(1-4):38--46, 1927/28.

\bibitem[Bla16]{Bla}
Iaroslav~V. Blagouchine.
\newblock Two series expansions for the logarithm of the gamma function
  involving {S}tirling numbers and containing only rational coefficients for
  certain arguments related to {$\pi^{-1}$}.
\newblock {\em J. Math. Anal. Appl.}, 442(2):404--434, 2016.

\bibitem[Bra06]{Bra}
David Branson.
\newblock Stirling number representations.
\newblock {\em Discrete Math.}, 306(5):478--494, 2006.

\bibitem[Bre93]{Bre}
Francesco Brenti.
\newblock Permutation enumeration symmetric functions, and unimodality.
\newblock {\em Pacific J. Math.}, 157(1):1--28, 1993.

\bibitem[Cha02]{chacha}
Charalambos~A. Charalambides.
\newblock {\em Enumerative combinatorics}.
\newblock CRC Press Series on Discrete Mathematics and its Applications.
  Chapman \& Hall/CRC, Boca Raton, FL, 2002.

\bibitem[Chu97]{Chu97}
Wenchang Chu.
\newblock Hypergeometric series and the {R}iemann zeta function.
\newblock {\em Acta Arith.}, 82(2):103--118, 1997.

\bibitem[Chu19]{Chu19}
Wenchang Chu.
\newblock Logarithms of a binomial series: extension of a series of {K}nuth.
\newblock {\em Math. Commun.}, 24(1):83--90, 2019.

\bibitem[Com74]{Comtet}
Louis Comtet.
\newblock {\em Advanced combinatorics}.
\newblock D. Reidel Publishing Co., Dordrecht, enlarged edition, 1974.
\newblock The art of finite and infinite expansions.

\bibitem[DM97]{dem}
Abraham De~Moivre.
\newblock A method of raising an infinite multinomial to any given power, or
  extracting any given root of the same.
\newblock {\em Philos. Trans. R. Soc. London}, 19(230):619–--625, 1697.
\newblock Also in {\em Miscellanea analytica de seriebus et quadraturis}, J.
  Tonson \& J. Watts, London, 1730.

\bibitem[Egg19]{Egg}
Eric~S. Egge.
\newblock {\em An introduction to symmetric functions and their combinatorics},
  volume~91 of {\em Student Mathematical Library}.
\newblock American Mathematical Society, Providence, RI, [2019] \copyright
  2019.

\bibitem[EJ]{EJ}
Mark Elin and Fiana Jacobzon.
\newblock Families of inverse functions: coefficient bodies and the
  {F}ekete--{S}zeg\"{o} problem.
\newblock arXiv:2012.07153.

\bibitem[FS95]{FS95}
Philippe Flajolet and Robert Sedgewick.
\newblock Mellin transforms and asymptotics: finite differences and {R}ice's
  integrals.
\newblock volume 144, pages 101--124. 1995.
\newblock Special volume on mathematical analysis of algorithms.

\bibitem[Fun30]{Funk}
H.~Gray Funkhouser.
\newblock A {S}hort {A}ccount of the {H}istory of {S}ymmetric {F}unctions of
  {R}oots of {E}quations.
\newblock {\em Amer. Math. Monthly}, 37(7):357--365, 1930.

\bibitem[Ges16]{Ge16}
Ira~M. Gessel.
\newblock Lagrange inversion.
\newblock {\em J. Combin. Theory Ser. A}, 144:212--249, 2016.

\bibitem[GKP94]{Knu}
Ronald~L. Graham, Donald~E. Knuth, and Oren Patashnik.
\newblock {\em Concrete mathematics}.
\newblock Addison-Wesley Publishing Company, Reading, MA, second edition, 1994.
\newblock A foundation for computer science.

\bibitem[Gou99]{Gou}
Henry~W. Gould.
\newblock The {G}irard-{W}aring power sum formulas for symmetric functions and
  {F}ibonacci sequences.
\newblock {\em Fibonacci Quart.}, 37(2):135--140, 1999.

\bibitem[Hof17]{Hoff}
Michael~E. Hoffman.
\newblock Harmonic-number summation identities, symmetric functions, and
  multiple zeta values.
\newblock {\em Ramanujan J.}, 42(2):501--526, 2017.

\bibitem[Jha]{Jhatri}
Sumit~Kumar Jha.
\newblock An identity involving number of representations of $n$ as a sum of
  $r$ triangular numbers.
\newblock arXiv:2011.11038.

\bibitem[Jha21]{Jha21}
Sumit~Kumar Jha.
\newblock A formula for the number of overpartitions of {$n$} in terms of the
  number of representations of {$n$} as a sum of {$r$} squares.
\newblock {\em Integers}, 21:Paper No. A82, 3, 2021.

\bibitem[Knu92]{K2}
Donald~E. Knuth.
\newblock Two notes on notation.
\newblock {\em Amer. Math. Monthly}, 99(5):403--422, 1992.

\bibitem[Kon00]{Kon}
John Konvalina.
\newblock A unified interpretation of the binomial coefficients, the {S}tirling
  numbers, and the {G}aussian coefficients.
\newblock {\em Amer. Math. Monthly}, 107(10):901--910, 2000.

\bibitem[Mac60]{Macm}
Percy~A. MacMahon.
\newblock {\em Combinatory analysis}.
\newblock Chelsea Publishing Co., New York, 1960.

\bibitem[Mac95]{Macd}
Ian~G. Macdonald.
\newblock {\em Symmetric functions and {H}all polynomials}.
\newblock Oxford Mathematical Monographs. The Clarendon Press, Oxford
  University Press, New York, second edition, 1995.
\newblock With contributions by A. Zelevinsky, Oxford Science Publications.

\bibitem[Mor71]{Mo71}
Alun~O. Morris.
\newblock Generalizations of the {C}auchy and {S}chur identities.
\newblock {\em J. Combinatorial Theory Ser. A}, 11:163--169, 1971.

\bibitem[Mor17]{Mort}
Eric~T. Mortenson.
\newblock A {K}ronecker-type identity and the representations of a number as a
  sum of three squares.
\newblock {\em Bull. Lond. Math. Soc.}, 49(5):770--783, 2017.

\bibitem[MR15]{MR15}
Anthony Mendes and Jeffrey Remmel.
\newblock {\em Counting with symmetric functions}, volume~43 of {\em
  Developments in Mathematics}.
\newblock Springer, Cham, 2015.

\bibitem[O'S]{odm}
Cormac O'Sullivan.
\newblock {D}e {M}oivre and {B}ell polynomials.
\newblock arxiv.

\bibitem[QSLK17]{Qi17}
Feng Qi, Xiao-Ting Shi, Fang-Fang Liu, and Dmitry~V. Kruchinin.
\newblock Several formulas for special values of the {B}ell polynomials of the
  second kind and applications.
\newblock {\em J. Appl. Anal. Comput.}, 7(3):857--871, 2017.

\bibitem[Rad73]{Ra}
Hans Rademacher.
\newblock {\em Topics in analytic number theory}.
\newblock Springer-Verlag, New York, 1973.
\newblock Edited by E. Grosswald, J. Lehner and M. Newman, Die Grundlehren der
  mathematischen Wissenschaften, Band 169.

\bibitem[Saa93]{Saa}
Louis Saalsch\"{u}tz.
\newblock {\em Vorlesungen \"{u}ber die {B}ernoullischen {Z}ahlen: ihren
  {Z}usammenhang mit den {S}ecanten-{C}oefficienten und ihre wichtigeren
  {A}nwendungen}.
\newblock Springer, Berlin, 1893.

\bibitem[Ses17]{Roman}
Javier Sesma.
\newblock The {R}oman harmonic numbers revisited.
\newblock {\em J. Number Theory}, 180:544--565, 2017.

\bibitem[Sta99]{Sta}
Richard~P. Stanley.
\newblock {\em Enumerative combinatorics. {V}ol. 2}, volume~62 of {\em
  Cambridge Studies in Advanced Mathematics}.
\newblock Cambridge University Press, Cambridge, 1999.
\newblock With a foreword by Gian-Carlo Rota and appendix 1 by Sergey Fomin.

\bibitem[Ste43]{Ste}
Moritz~A. Stern.
\newblock Ueber die {C}o\"{e}fficienten der {S}ecantenreihe.
\newblock {\em J. Reine Angew. Math.}, 26:88--91, 1843.

\bibitem[Sun05]{Sun05}
Zhi-Hong Sun.
\newblock On the properties of {N}ewton-{E}uler pairs.
\newblock {\em J. Number Theory}, 114(1):88--123, 2005.

\end{thebibliography}

{\small %\footnotesize
\vskip 5mm
\noindent
\textsc{Dept. of Math, The CUNY Graduate Center, 365 Fifth Avenue, New York, NY 10016-4309, U.S.A.}

\noindent
{\em E-mail address:} \texttt{cosullivan@gc.cuny.edu}
}

\end{document}